\documentclass[12pt]{amsart}
\usepackage{amsthm,latexsym,amssymb}
\usepackage{hyperref}
\usepackage[centertags]{amsmath}
\usepackage{mathrsfs}
\usepackage{t1enc}
\usepackage{graphicx}
\usepackage{enumerate}

\advance\headheight by 3pt

\newtheorem{thm}{Theorem}[section]
\newtheorem{cor}[thm]{Corollary}
\newtheorem{lem}[thm]{Lemma}

\newtheorem*{theorem*}{Theorem}
\newtheorem{prop}[thm]{Proposition}  

\newtheorem{thmalpha}{Theorem}
  
\numberwithin{equation}{section}
\def\Q{{\mathbb Q}}

\def\Z{{\mathbb Z}}

\def\vec#1{{\bf #1}}

\def\e{\textrm{e}}

\def\pitem{\advance\leftskip3mm\advance\linewidth-3mm}
\def\mitem{\advance\leftskip-3mm\advance\linewidth3mm}

\begingroup\makeatletter\ifx\SetFigFont\undefined
\gdef\SetFigFont#1#2#3#4#5{
  \reset@font\fontsize{#1}{#2pt}
  \fontfamily{#3}\fontseries{#4}\fontshape{#5}
  \selectfont}
\fi\endgroup

\renewcommand{\subjclass}[1]{\thanks{\emph{2010 Mathematics Subject Classification:}~#1}}
\renewcommand{\keywords}[1]{\thanks{\emph{Keywords and Phrases:}~#1}}
\renewcommand{\date}{\thanks{\today}}

\newcommand{\av}{{\bf a}}
\newcommand{\bv}{{\bf b}}

\newcommand{\uv}{{\bf u}}

\newcommand{\xv}{{\bf x}}
\newcommand{\yv}{{\bf y}}

\newcommand{\CC}{\mathcal{C}}

\newcommand{\GG}{\mathcal{G}}

\newcommand{\II}{\mathcal{I}}
\newcommand{\JJ}{\mathcal{J}}

\newcommand{\LL}{\mathcal{L}}
\newcommand{\MM}{\mathcal{M}}
\newcommand{\NN}{\mathcal{N}}

\newcommand{\PP}{\mathcal{P}}

\newcommand{\fp}{\mathfrak{p}}

\newcommand{\fa}{\mathfrak{a}}

\newcommand{\Cc}{\mathbb{C}}
\newcommand{\Rr}{\mathbb{R}}
\newcommand{\Qq}{\mathbb{Q}}

\newcommand{\Zz}{\mathbb{Z}}

\renewcommand{\ge}{\geq}
\renewcommand{\le}{\leq}

\def\house#1{\setbox1=\hbox{$\,#1\,$}%
\dimen1=\ht1 \advance\dimen1 by 2pt \dimen2=\dp1 \advance\dimen2
by 2pt
\setbox1=\hbox{\vrule height\dimen1 depth\dimen2\box1\vrule}%
\setbox1=\vbox{\hrule\box1}%
\advance\dimen1 by .4pt \ht1=\dimen1 \advance\dimen2 by .4pt
\dp1=\dimen2 \box1\relax}

\newcommand{\propersubset}
{\,\mbox{{\raisebox{-1ex}
{$\stackrel{\textstyle{\subset}}{\scriptscriptstyle{\not=}}$}}}\,}

\newcommand{\kdots}{,\ldots ,}

\newcommand{\half}{\mbox{$\textstyle{\frac{1}{2}}$}}

\newcommand{\medfrac}[2]{\mbox{\large{$\textstyle{\frac{#1}{#2}}$}}}

\renewcommand{\mod}[3]{#1\equiv#2\,({\rm mod}\,#3)}
\newcommand{\nmod}[3]{#1\not\equiv#2\,({\rm mod}\,#3)}
\renewcommand{\gcd}{{\rm gcd}}
\newcommand{\lcm}{{\rm lcm}}
\newcommand{\galk}{{\rm Gal}(K/\mathbb{Q})}
\newcommand{\zmprim}{\mathbb{Z}^m_{{\rm prim}}}
\newcommand{\zprim}{\mathbb{Z}^2_{{\rm prim}}}
\newcommand{\rank}{{\rm rank}\,}
\newcommand{\rankD}{{\rm rank}_D\,}
\newcommand{\Si}{S\cup\{\infty\}}

\title[S-parts of values of polynomials]
{S-parts of values of univariate polynomials, binary forms and decomposable forms at integral points}
\date

\author[Y. Bugeaud]{Yann Bugeaud}
\address{Y. Bugeaud \newline
         \indent IRMA U.M.R. 7501, Universit\'{e} de Strasbourg et CNRS, \newline
         \indent 7, rue Ren\'{e} Descartes, F-67084 Strasbourg cedex, France}
\email{bugeaud\char'100math.unistra.fr}

\author[J.-H. Evertse]{Jan-Hendrik Evertse}
\address{J.-H. Evertse \newline
         \indent Universiteit Leiden, Mathematisch Instituut, \newline
         \indent Postbus 9512, 2300 RA Leiden, The Netherlands}
\email{evertse\char'100math.leidenuniv.nl}

\author[K. Gy\H{o}ry]{K\'{a}lm\'{a}n Gy\H{o}ry}
\address{K. Gy\H{o}ry \newline
         \indent Institute of Mathematics, University of Debrecen \newline
         \indent H-4002 Debrecen, P.O. Box 400, Hungary}
\email{gyory\char'100science.unideb.hu}

\begin{document}

\maketitle

\vspace{-0.25cm}
\begin{center}
\large{\emph{To Robert Tijdeman on his 75-th birthday}}
\end{center}

\section{Introduction}\label{sec0}

Let $S = \{p_1, \ldots , p_s\}$ be a finite, non-empty set of distinct prime numbers.
For a non-zero integer $m$, write $m = p_1^{a_1} \ldots p_s^{a_s} b$, where 
$a_1, \ldots , a_s$ are non-negative integers and $b$ is an integer 
relatively prime to $p_1 \cdots p_s$. Then we define the $S$-part $[m]_S$ 
of $m$ by 
$$
[m]_S := p_1^{a_1} \ldots p_s^{a_s}.
$$
The motivation of the present paper was given by
the following result, established in 2013 by Gross and Vincent \cite{grovin}.

\begin{thmalpha}\label{A}
Let $f(X)$ be a polynomial with integral coefficients
with at least two distinct roots 
and $S$ a finite, non-empty set of prime numbers. 
Then there exist effectively computable positive numbers $\kappa_1$ and $\kappa_2$,
depending only on $f(X)$ and $S$, such that for every non-zero integer $x$ that is not 
a root of $f(X)$ we have
$$
[f(x)]_S < \kappa_2 |f(x)|^{1 - \kappa_1}.
$$
\end{thmalpha}

Gross and Vincent's proof of Theorem \ref{A} depends on the theory of linear forms in complex logarithms,
Under the additional hypotheses that $f(X)$ has  
degree $n\geq 2$ and no multiple roots,
we deduce an ineffective analogue of Theorem \ref{A}, with instead of $1-\kappa_1$ an
exponent $\medfrac{1}{n}+\epsilon$ for every $\epsilon >0$ and instead
of $\kappa_2$ an ineffective number depending on $f(X)$,
$S$ and $\epsilon$. 
This is in fact an easy application of the $p$-adic Thue-Siegel-Roth 
Theorem.
We show that the exponent
$\medfrac{1}{n}$ is best possible. Lastly, we give an estimate
for the density of the set of integers $x$ for which $[f(x)]_S$
is large, i.e., for every small $\epsilon >0$ we estimate in terms of $B$
the number of integers $x$ with $|x|\leq B$
such that $[f(x)]_S\geq |f(x)|^{\epsilon}$.

We considerably extend both Theorem \ref{A}, its ineffective analogue,
and the density result
by proving similar results 
for the $S$-parts of values of homogeneous binary forms and, more generally, of 
values of decomposable forms at integer points, under suitable assumptions. In addition, in the effective results
we give an expression for $\kappa_1$, which is explicit in terms of $S$. 
For our extensions to binary forms and decomposable forms,
we use the $p$-adic Thue-Siegel-Roth Theorem
and the $p$-adic Subspace Theorem of Schmidt and Schlickewei for the ineffective estimates for the $S$-part. 
The proof of the effective estimates is based on
an effective theorem of Gy\H{o}ry and Yu \cite{gyyu} 
on decomposable form equations 
whose proof depends on estimates for linear 
forms in complex and in $p$-adic logarithms. Lastly,
the proofs of our density results on the number of integer points
of norm at most $B$ at which the value of the binary form or
decomposable form under consideration has large $S$-value are
based on a recent general lattice point counting result
of Barroero and Widmer \cite{BarWid14} and on work in the 
PhD-thesis of Junjiang Liu \cite{liu15}.  

For simplicity, we have restricted ourselves to univariate
polynomials, binary forms and decomposable forms with coefficients in $\Zz$. 
With some extra technical effort, analogous results could have been
obtained over arbitrary number fields.

In Section \ref{sec1} we state our results, in Sections \ref{section3}--\ref{section6} we give the proofs,
in Sections \ref{section7} and \ref{section8} we present some applications,
and in Section \ref{section-last} we give some additional comments
on Theorem \ref{A}. 

\section{Results}\label{sec1}

\subsection{Results for univariate polynomials and binary forms}\label{sec1a}

We use notation $\ll_{a,b,\ldots}$, $\gg_{a,b,\ldots}$ to indicate
that the constants implied by the Vinogradov symbols depend only
on the parameters $a,b,\ldots$ .
Further, we use the notation $A\asymp_{a,b,\ldots}B$ to denote 
$A\ll_{a,b,\ldots} B$ and $B\ll_{a,b,\ldots} A$.
We prove the following ineffective analogue of
Theorem \ref{A} mentioned in the previous section.

\begin{thm}\label{thm1a.1}
Let $f(X)\in\Zz [X]$ be a polynomial of degree $n\geq 2$ without multiple
zeros.
\begin{itemize}
\item[(i)] Let $S=\{ p_1\kdots p_s\}$ be a non-empty set of primes.
Then for every $\epsilon >0$ and
for every $x\in\Zz$ with $f(x)\not= 0$,
\[
[f(x)]_S\ll_{f,S,\epsilon} |f(x)|^{(1/n)+\epsilon}. 
\]
\item[(ii)] There are infinitely many primes $p$, and for every of these
$p$, there are infinitely many integers $x$, such that
$f(x)\not= 0$ and 
\[
[f(x)]_{\{ p\}}\gg_f |f(x)|^{1/n}.
\]
\end{itemize}
\end{thm}

For completeness, we give here
also a more precise effective version of Theorem \ref{A}, 
which is a consequence of Theorem \ref{thm1b.3} stated below
on the $S$-parts of values of
binary forms.

\begin{thm}\label{thm1a.3}
Let $f(X)\in\Zz [X]$ be a polynomial
with at least two distinct roots and suppose that its splitting field
has degree $d$ over $\Qq$. Further, let
$S=\{ p_1\kdots p_s\}$ be a non-empty set of primes and put $P:=\max (p_1\kdots p_s)$. Then for every integer $x$ with $f(x)\not= 0$ we have
\[
[f(x)]_S\leq \kappa_2|f(x)|^{1-\kappa_1},
\]
where
\[
\kappa_1=\Big(c_1^s\big(P(\log p_1)\cdots (\log p_s)\big)^d\Big)^{-1},
\]
and $c_1,\kappa_2$ are effectively computable positive numbers that depend
only on $f(X)$.
\end{thm}

For variations on this result, and related results,
we refer to Section \ref{section-last}. 

For polynomials $X(X+1)$ and $X^2 + 7$ and special sets $S$, Bennett, Filaseta,     
and Trifonov \cite{BeFiTr08,BeFiTr09} have obtained stronger effective results.     

As is to be expected, for most integers $x$, the $S$-part $[f(x)]_S$ is small
compared with $|f(x)|$. This is made more precise in the following result.
For any finite set of primes $S$ and any $\epsilon >0$, $B>0$,
we denote by $N(f,S,\epsilon ,B)$ the number of integers $x$ such that
\begin{equation}\label{1a.1}
|x|\leq B,\ \ f(x)\not=0,\ \  [f(x)]_S\geq |f(x)|^{\epsilon}.
\end{equation}
Denote by $D(f)$ the discriminant of $f$ and for a prime $p$,
denote by $g_p$ the largest integer $g$ such that
$p^{g}$ divides $D(f)$.

\begin{thm}\label{thm1a.2}
Let $f(X)\in\Zz [X]$ be a polynomial of degree $n\geq 2$ with
non-zero discriminant. Further, let
$0<\epsilon <1/n$, and let $S$ be a finite set of primes.
Denote by $s'$ the number of primes $p\in S$ such that
$\mod{f(x)}{0}{p^{g_p +1}}$ is solvable and assume that this
number is positive. 
Then
\[
N(f,S,\epsilon ,B)\asymp_{f,S,\epsilon} B^{1-n\epsilon}(\log B)^{s'-1}\ \ \mbox{as } B\to\infty.
\]
\end{thm}

\noindent
{\bf Remarks.}
\\
{\bf 1.} 
If $s'=0$ then $[f(x)]_S$ is bounded,
and so the set of integers $x$ with $[f(x)]_S\geq |f(x)|^{\epsilon}$
is finite. 
\\[0.15cm]
{\bf 2.} In general, $\lim_{B\to\infty} N(f,S,\epsilon ,B)/B^{1-n\epsilon}(\log B)^{s'-1}$ does not exist. 
\\[0.15cm]
{\bf 3.} There are infinitely many
primes $p$ such that $\mod{f(x)}{0}{p}$ is solvable.
Removing from those the finitely many that divide $D(f)$, there
remain infinitely many primes $p$ such that $g_p=0$ and 
$\mod{f(x)}{0}{p}$ is solvable.
\\

We now formulate some analogues of the above mentioned results for
binary forms. Denote by $\zprim$ the set of pairs $(x,y)\in\Zz^2$ with
$\gcd (x,y)=1$.

\begin{thm}\label{thm1b.1}
Let $F(X,Y)\in\Zz [X,Y]$ be a binary form of degree $n\geq 2$ with non-zero
discriminant.
\begin{itemize}
\item[(i)] Let $S=\{p_1\kdots p_s\}$ be a non-empty set of primes.
Then for every $\epsilon >0$ and every pair $(x,y)\in\zprim$ with $F(x,y)\not= 0$,
\[
[F(x,y)]_S\ll_{F,S,\epsilon} |F(x,y)|^{(2/n)+\epsilon}.
\]
\item[(ii)] There are finite sets of primes $S$ with 
the smallest prime in $S$ arbitrarily large,
and for every of these $S$ infinitely many pairs $(x,y)\in\zprim$, such that
$F(x,y)\not= 0$ and
\[
[F(x,y)]_S\gg_{F,S,\epsilon} |F(x,y)|^{2/n}.
\]
\end{itemize}
\end{thm}

Our next result is an effective analogue of Theorem \ref{thm1a.3} for binary forms.
It is an easy consequence of Theorem \ref{thm1}
stated below on decomposable forms. The splitting field of a binary form is the smallest
extension of $\Qq$ over which it factors into linear forms.

\begin{thm}\label{thm1b.3}
Let $F(X,Y)$ be a binary form of degree $n\ge 3$ with coefficients in $\Zz$ 
and with splitting field $K$. Suppose that $F$ has at least three pairwise non-proportional linear factors over $K$. 
Let again $S=\{ p_1\kdots p_s\}$ be a finite set of primes
and $[K:\Qq ]=d$.
Then
\[ 
[F(x,y)]_S\le \kappa_4 |F(x, y)|^{1-\kappa_3}       
\]
for every $(x,y)\in\zprim$ with $F(x,y)\not=0$,
where 
$$
\kappa_3= 
\bigl(c_2^s \bigl((P (\log p_1) \cdots (\log p_s) \bigr)^d \bigr)^{-1} 
$$ 
and 
$\kappa_4$, $c_2$  are effectively computable positive numbers, 
depending only on $F$.
\end{thm}

We obtain Theorem \ref{thm1a.3} on polynomials $f(X)\in\Zz [X]$ by 
applying Theorem \ref{thm1b.3} to the binary form $Y^{1 + \deg f}f(X/Y)$     
with $(x,y)=(x,1)\in\zprim$.

Let again $F(X,Y)\in\Zz [X,Y]$ be a binary form of degree $n\geq 2$
and of non-zero discriminant.
For any finite set of primes $S$ and any $\epsilon >0$, $B>0$,
we denote by $N(F,S,\epsilon ,B)$ the number of pairs $(x,y)\in\zprim$ such that
\begin{equation}\label{1a.1a}
\max(|x|,|y|)\leq B,\ \ F(x,y)\not=0,\ \  [F(x,y)]_S\geq |F(x,y)|^{\epsilon}.
\end{equation}
Denote by $D(F)$ the discriminant of $F$ and for a prime $p$,
denote by $g_p$ the largest integer $g$ such that
$p^{g}$ divides $D(F)$.

\begin{thm}\label{thm1b.2}
Let $F(X,Y)\in\Zz [X,Y]$ be a binary form of degree $n\geq 3$ with
non-zero discriminant. Further, let
$0<\epsilon <\medfrac{1}{n}$, and let $S$ be a finite set of primes.
Denote by $s'$ the number of primes $p\in S$ such that
$\mod{F(x,y)}{0}{p^{g_p +1}}$ has a solution $(x,y)\in\zprim$
and assume that this
number is positive. 
Then
\[
N(F,S,\epsilon ,B)\asymp_{F,S,\epsilon} B^{2-n\epsilon}(\log B)^{s'-1}\ \ \mbox{as } B\to\infty.
\]
\end{thm}

Parts (i) of Theorems \ref{thm1a.1} and \ref{thm1b.1} are easy consequences
of the $p$-adic Thue-Siegel-Roth Theorem. Part (ii) of Theorem \ref{thm1a.1} is a
consequence of the fact that for a given non-constant polynomial $f(X)\in\Zz [X]$
there are infinitely many primes $p$ such that $f(X)$ has a zero in $\Zz_p$.
The proof of part (ii) of Theorem \ref{thm1b.1} uses some geometry of numbers.

There are two main tools in the proof of Theorem \ref{thm1b.2}.
The first is a result of Stewart
\cite[Thm. 2]{stewart} on the number of congruence classes
$x$ modulo $p^k$ of $\mod{f(x)}{0}{p^k}$ for $f(X)$ a polynomial
and $p^k$ a prime power.
The second is a powerful lattice point counting result
of Barroero and Widmer \cite[Thm. 1.3]{BarWid14}.
The proof of Theorem \ref{thm1a.2} is very similar,
but instead of the result of Barroero and Widmer it uses a much more
elementary counting argument.

\subsection{Ineffective results for decomposable forms}\label{sec1b}

We will state results on the $S$-parts of values of decomposable forms
in $m$ variables at integral points, where $m\geq 2$.

We start with some notation and definitions.
Let $K$ be a finite, normal extension of $\Qq$.
For a linear form $\ell =\alpha_1X_1+\cdots +\alpha_mX_m$
with coefficients in $K$ and for an element $\sigma$ of the Galois group
$\galk$ we define $\sigma (\ell ):=\sigma (\alpha_1)X_1+\cdots +\sigma (\alpha_m)X_m$ and then for a set of linear forms $\LL=\{\ell_1\kdots\ell_r\}$
with coefficients in $K$ we write 
$\sigma (\LL ):=\{ \sigma (\ell_1)\kdots \sigma (\ell_r )\}$. 
A set of linear forms $\LL$ with coefficients in $K$ is called
\emph{$\galk$-symmetric} if $\sigma (\LL )=\LL$ for each $\sigma\in\galk$,
and \emph{$\galk$-proper} if for each $\sigma\in\LL$ we have either
$\sigma (\LL )=\LL$ or $\sigma (\LL )\cap\LL =\emptyset$.
We denote by $[\LL ]$ the $K$-vector space generated by $\LL$,
and define $\rank\LL$ to be the dimension of $[\LL ]$ over $K$.
Finally, we define the sum of two vector spaces $V_1$, $V_2$ over $K$ by
$V_1+V_2:=\{ \xv +\yv :\, \xv\in V_1,\, \yv\in V_2\}$.

Recall that a decomposable form in $\Zz [X_1\kdots X_m]$
is a homogeneous polynomial that factors into linear forms in $X_1\kdots X_m$
over some
extension of $\Qq$. The smallest extension over which such a factorization
is possible is called the 
\emph{splitting field} 
of the decomposable form.
This is a finite, normal extension of $\Qq$. 

Let $F\in\Z[X_1\kdots X_m]$ be a decomposable form of degree $n\ge 3$
with splitting field $K$. Then we can express $F$ as
\begin{equation}\label{eq1.0}
\left\{
\begin{array}{l}
F=c\ell_1^{e(\ell_1)}\cdots \ell_r^{e(\ell_r)}\ \ \mbox{with} 
\\[0.15cm]
\mbox{$c$ a non-zero rational,}
\\[0.15cm]
\mbox{$\LL_F=\{ \ell_1\kdots\ell_r\}$
a $\galk$-symmetric set of pairwise}
\\
\mbox{non-proportional linear forms with coefficients in $K$,}
\\[0.15cm]
\mbox{$e(\ell_1)\kdots e(\ell_r)$ positive integers,
with $e(\ell_i)=e(\ell_j)$}
\\
\mbox{whenever $\ell_j=\sigma (\ell_i)$ for some $\sigma\in\galk$.}
\end{array}\right.
\end{equation}
Lastly, define $\zmprim$ to be the set of ${\bf x}=(x_1\kdots x_m)\in\Zz^m$
with $\gcd (x_1\kdots x_m)=1$ and define $\|\xv\|$ to be the maximum norm
of $\xv\in\zmprim$.

Let $S=\{ p_1\kdots p_s\}$ be a finite set of primes,
and $F\in\Zz [X_1\kdots X_m ]$ a decomposable form.
For $\xv\in\zmprim$ with $F(\xv )\not= 0$, we can write
\begin{equation}\label{eq:1}
F(\xv )=p_1^{a_1}\cdots p_s^{a_s}\cdot b,
\end{equation}
where $a_1,\ldots,a_s$ are non-negative integers and $b$ is an integer
coprime with $p_1\cdots p_s$. Then the $S$-part
$[F(\xv )]_S$ is $p_1^{a_1}\cdots p_s^{a_s}$.
We may view \eqref{eq:1} as a Diophantine equation in $\xv\in\zmprim$
and $a_1\kdots a_s\in\Zz_{\geq 0}$, a so-called
\emph{decomposable form equation}.
Schlickewei \cite{schl77} considered \eqref{eq:1} in the case that $F$ is a norm
form (i.e., a decomposable form that is irreducible over $\Qq$) and formulated
a criterion in terms of $F$ implying that \eqref{eq:1} has only finitely many 
solutions. Evertse and Gy\H{o}ry \cite{evgy88} gave another finiteness criterion in terms of $F$, valid for arbitrary decomposable forms.
Recently
\cite[Chap. 9, Thm. 9.1.1]{evgy15}, they refined this as follows.
Call an integer \emph{$S$-free} if it is non-zero, and coprime with the primes in $S$.

\begin{thmalpha}\label{B}
Let $F\in\Zz [X_1\kdots X_m]$ be a decomposable form with splitting field $K$,
given in the form \eqref{eq1.0}, and let $\LL$ be a finite set of linear forms
in $K[X_1\kdots X_m]$, containing $\LL_F$. Then the following two assertions are equivalent:
\begin{itemize}
\item[(i)] $\rank\LL_F=m$, and for every $\galk$-proper subset $\MM$
of $\LL_F$ with $\emptyset\propersubset\MM\propersubset\LL_F$, we have
\begin{equation}\label{eq1.0a}
\LL\cap\Big(\sum_{\sigma\in\galk} [\sigma (\MM )]\cap[\LL_F\setminus\sigma (\MM )]\Big)\not=\emptyset ;
\end{equation}
\item[(ii)] for every finite set of primes $S=\{ p_1\kdots p_s\}$ and 
every $S$-free integer $b$, there are only finitely many $\xv\in\zmprim$
and non-negative integers $a_1\kdots a_s$ such that
\begin{equation}\label{eq1.0b}
F(\xv )= p_1^{a_1}\cdots p_s^{a_s}b,\ \ \ \ell (\xv )\not=0\ \mbox{for } \ell\in\LL .
\end{equation}
\end{itemize}
\end{thmalpha}

\noindent
This theorem was deduced from a finiteness theorem of Evertse \cite{ev84}
and van der Poorten and Schlickewei \cite{vdpschl82,vdpschl91}
on $S$-unit equations over number fields. 

The following result gives an improvement of (ii).
We denote by $|\cdot |_{\infty}$ the standard archimedean absolute value on $\Qq$,
and for a prime $p$ by $|\cdot |_p$ the standard $p$-adic absolute value, with 
$|p|_p=p^{-1}$. Further, $\|\xv\|$ denotes the maximum norm
of $\xv\in\zmprim$. 

\begin{thm}\label{thm2.0}
Let $F\in\Zz [X_1\kdots X_m]$ be a decomposable form in $m\geq 2$ variables
with splitting field $K$ and $\LL\supseteq\LL_F$ a finite set of linear forms
in $K[X_1\kdots X_m]$, satisfying condition (i) of Theorem \ref{B}.
Further, let $S$ be a finite set of primes and let 
$\epsilon >0$.
Then there are only finitely many $\xv\in\zmprim$
with 
\begin{equation}\label{eq1.0bb}
\left\{\begin{array}{l}
\displaystyle{\prod_{p\in\Si} |F(\xv )|_p\leq \|\xv\|^{(1/(m-1))-\epsilon}},\\[0.3cm] 
\quad\ell (\xv )\not= 0\ \mbox{for } \ell\in\LL .
\end{array}\right. 
\end{equation}
\end{thm}

Chen and Ru \cite{chenru} proved a similar result with $\LL =\LL_F$
the set of linear factors of $F$ and with a stronger
condition instead of (i), on the other hand 
they considered decomposable forms with coefficients 
in an arbitrary number field.

From Theorem \ref{thm2.0} and Theorem \ref{B} we deduce the following corollary.

\begin{cor}\label{cor2.0} 
Let $F\in\Zz [X_1\kdots X_m]$ be a 
decomposable form in $m\geq 2$ variables
with splitting field $K$ and $\LL\supseteq\LL_F$ a finite set of linear forms
in $K[X_1\kdots X_m]$.
\begin{itemize}
\item[(i)]
Assume that $F$ and $\LL$ satisfy condition (i) of Theorem \ref{B}.
Suppose $F$ has degree $n$. Let $S$ be a finite set of primes and let 
$\epsilon >0$.
Then for every $\xv\in\zmprim$ with $\ell (\xv )\not= 0$ for $\ell\in\LL$
we have
\begin{equation}\label{eq1.0c}
[F(\xv )]_S
\ll_{F,\LL,S,\epsilon}|F(\xv )|^{1-(1/n(m-1))+\epsilon}.
\end{equation}
\item[(ii)]
Assume that $F$ and $\LL$ do not satisfy condition (i) of Theorem \ref{B}.
Then there are a finite set of primes $S$ and a constant $\gamma >0$ such that
\[
[F(\xv )]_S\geq \gamma |F(\xv )|
\]
holds for infinitely many $\xv\in\zmprim$ with $\ell (\xv )\not= 0$ 
for all $\ell\in\LL$.
\end{itemize}
\end{cor}

\noindent
Indeed, if $F$, $\LL$ satisfy condition (i) of Theorem \ref{B}, $S$ is a finite set
of primes and $\epsilon >0$ then
\[
\frac{|F(\xv )|}{[F(\xv )]_S}=\prod_{p\in \Si} |F(\xv )|_p
\gg\|\xv\|^{(1/(m-1))-\epsilon}
\gg |F(\xv )|^{(1/n(m-1))-\epsilon /n}
\]
for all $\xv\in\zmprim$ with $\ell (\xv )\not= 0$ for all $\ell\in\LL$,
where the implied constants depend on $F$, $S$ and $\epsilon$.
This implies part (i)
of Corollary \ref{cor2.0}. If on the other hand $F$ and $\LL$ do not satisfy
condition (i) of Theorem \ref{B} then there are a finite set of primes $S$
and an $S$-free integer $b$ such that \eqref{eq1.0b} has infinitely many solutions.
This yields infinitely many $\xv\in\zmprim$ such that $\ell (\xv )\not=0$
for all $\ell\in\LL$ and
\[
[F(\xv )]_S=|F(\xv )|/|b|.
\]
Thus, part (ii) of Corollary \ref{cor2.0} follows.
\\[0.2cm]

We can improve on Corollary \ref{cor2.0} if we assume condition (i) of Theorem \ref{B}
with $\LL =\LL_F$, i.e., 
\begin{eqnarray}\label{eq1.0e} 
&&\mbox{$\rank\LL_F=m$, and for every $\galk$-proper subset}
\\[-0.12cm]
\nonumber
&&\mbox{$\MM$ of $\LL_F$ with 
$\emptyset\propersubset\MM\propersubset\LL_F$ we have}
\\[0.1cm]
\nonumber
&&\ \ \LL_F\cap\Big(\sum_{\sigma\in\galk} [\sigma (\MM )]\cap[\LL_F\setminus\sigma (\MM )]\Big)\not=\emptyset
\end{eqnarray}
and in addition to this,
\begin{equation}\label{eq1.0d}
F(\xv )\not=0\ \ \mbox{for every non-zero $\xv\in\Qq^m$.}
\end{equation}

Let $D$ be a $\Qq$-linear subspace of $\Qq^m$. We say that a 
non-empty 
subset $\MM$
of $\LL_F$ is \emph{linearly dependent on $D$}
if there is a non-trivial
$K$-linear combination of the forms in $\MM$ that vanishes identically on $D$;
otherwise, $\MM$ is said to be \emph{linearly independent on $D$}.
Further, for a non-empty subset $\MM$ of $\LL_F$
we define $\rankD\MM$ to be the cardinality of 
a maximal subset of $\MM$ that is linearly independent on $D$,
and then
\[
q_D(\MM ):=\frac{\sum_{\ell\in\MM} e(\ell )}{\rankD\MM}.
\]
For instance, $\rankD\LL_F=\dim D$, so $q_D(\LL_F)=\deg F/\dim D$. Then put
\[
q_D(F):=\max\{ q_D(\MM ):\emptyset\propersubset\MM\propersubset\LL_F,\ \rankD\MM <\dim D\}.
\]
Finally, put
\[
c(F):=\max_D\frac{q_D(F)}{q_D(\LL_F)}=\max_D q_D(F)\cdot \frac{\dim D}{\deg F},
\]
where the maximum is taken over all $\Qq$-linear subspaces $D$ of $\Qq^m$
with $\dim D\geq 2$. 
Lemma \ref{lem5.2}, which is stated and proved in Section \ref{section5} below,
implies that if $F$ satisfies both \eqref{eq1.0e} and 
\eqref{eq1.0d}, then $c(F)<1$. 
We will not consider the problem
how to compute $c(F)$, that is, how to determine a subspace $D$
for which $q_D(F)/q_D(\LL_F)$ is maximal; 
this may involve some linear algebra that is
beyond the scope of this paper. 

Given a decomposable form $F\in\Zz [X_1\kdots X_m]$,
a finite set of primes $S$, and reals $\epsilon >0$, $B>0$, we define
$N(F,S,\epsilon ,B)$ to be the set of $\xv\in\zmprim$ with
$[F(\xv )]_S\geq |F(\xv )|^{\epsilon}$ and $\|\xv\|\leq B$.

\begin{thm}\label{thm2.1}
Let $m\geq 2$ and let $F\in\Zz [X_1\kdots X_m]$ be a decomposable form 
as in \eqref{eq1.0} satisfying \eqref{eq1.0e} and \eqref{eq1.0d}.
\begin{itemize}
\item[(i)] For every finite set of primes $S$, every $\epsilon >0$ and
every $\xv\in\zmprim$ we have
\[
[F(\xv )]_S\ll_{F,S,\epsilon} |F(\xv )|^{c(F)+\epsilon};
\]
\item[(ii)] There are infinitely many primes $p$, and for every of these
primes $p$ infinitely many $\xv\in\zmprim$, such that
\[
[F(\xv )]_{\{ p\}}\gg_{F,p} |F(\xv )|^{c(F)};
\]
\item[(iii)] For every finite set of primes $S$ and every $\epsilon$ 
with $0<\epsilon <1$ we have
\[
N(F,S,\epsilon ,B)\ll_{F,S,\epsilon} B^{m(1-\epsilon )}\ \ \mbox{as } B\to\infty .
\]
\end{itemize}
\end{thm}

Assertions (i) and (iii) follow without too much effort from work in
Liu's thesis \cite{liu15}, while (ii) is an application
of Minkowski's Convex Body Theorem.

The constants implied by the Vinogradov symbols in Theorems \ref{thm2.0} and part (i) of Theorem \ref{thm2.1} cannot be computed effectively
from our method of proof.
In fact, these constants can be expressed in terms of the heights of the subspaces occurring in certain instances
of the $p$-adic Subspace Theorem, but for these we can as yet not compute
an upper bound.
The constant in (ii) can be computed once one knows
a subspace $D$ for which the quotient
$q_D(F)/q_D(\LL_F)$ is equal to $c(F)$.
The work of Liu from which part (iii)
is derived uses a quantitative version of the $p$-adic Subspace
Theorem, giving an explicit upper bound for the \emph{number} of subspaces.
This enable one to compute effectively the constant in part (iii).

We mention that part (iii) of Theorem \ref{thm2.1} can be proved 
by a similar method as Theorem \ref{thm1b.3}, using the lattice point counting
result of Barroero and Widmer, thereby avoiding Liu's work and the
quantitative Subspace Theorem. But this approach would have been
much lengthier. 

\subsection{Effective results for decomposable forms}\label{sec1c}

We consider again $S$-parts of values $F(\xv )$, where $F$ is a decomposable form
in $\Zz [X_1\kdots X_m]$ and $\xv\in\zmprim$.
Under certain stronger conditions on $F$, we shall give an 
estimate of the form $[F(\xv )]_S\leq \kappa_6|F(\vec{x})|^{1-\kappa_5}$, with effectively computable positive $\kappa_5$, $\kappa_6$ that depend only on $F$ and $S$. For applications, we make the dependence of $\kappa_5$ and $\kappa_6$ 
explicit in terms of $S$. 
The decomposable forms with the said stronger conditions
include binary forms, and discriminant forms of an arbitrary number of
variables.

Let again $S=\{ p_1\kdots p_s\}$ be a finite set of primes and $b$ an integer
coprime with $p_1\cdots p_s$, and consider equation \eqref{eq:1}
in $\xv\in\zmprim$ and non-negative integers $a_1\kdots a_s$. 
Under the stronger conditions for the decomposable form $F$ mentioned above, 
explicit upper bounds were given in Gy\H ory \cite{gy1, gy2} for the solutions of 
\eqref{eq:1},
from which upper bounds can be deduced for $[F(\vec{x})]_S$. Later, more general and stronger explicit results were obtained by Gy\H ory and Yu \cite{gyyu} on another version of \eqref{eq:1}. These explicit results provided some information on the arithmetical properties of $F(\vec{x})$ at points $\vec{x}\in\zmprim$. In this paper, we deduce from the results of Gy\H ory and Yu \cite{gyyu} a better bound for  
$[F(\vec{x})]_S$; see Theorem \ref{thm1}. This will give more precise information on the arithmetical structure of those non-zero integers $F_0$ that can be represented by $F(\vec{x})$ at integral points $\vec{x}$; see Corollary \ref{cor2}.

To state our results, we introduce some notation and assumptions. 
Let $F\in\Zz [X_1\kdots X_m]$ be a non-zero decomposable form.
Denote by $K$ its splitting field.
We choose a factorization of $F$ into linear forms with coefficients in $K$
as in \eqref{eq1.0}, with $\LL_F$ a $\galk$-symmetric set 
of pairwise non-propertional linear forms.
Denote by $\GG (\LL_F)$ the graph with vertex set $\LL_F$ in which distinct $\ell$, $\ell'$ in $\LL_F$ are connected by an edge if $\lambda\ell+\lambda'\ell'+\lambda''\ell''=0$ for some $\ell''\in\LL_F$ and some non-zero $\lambda$, $\lambda'$, $\lambda''$ in $K$. Let $\mathcal{L}_1,\ldots,\mathcal{L}_k$ be the vertex sets of the connected components of $\GG (\LL_F)$. When $k=1$ and $\LL_F$ has at least three elements, $\LL_F$ is said to be \textit{triangularly connected}; see Gy\H ory and Papp \cite{gypapp}.

In what follows, we assume that $F$ in \eqref{eq:1} satisfies the following conditions:
\begin{eqnarray}
\label{eq:2a}
&&\mbox{$\LL_F$ \textit{has rank} $m$;}
\\
\label{eq:2b}
&&\mbox{\textit{either $k=1$; or $k>1$ and $X_m$ can be expressed as a}}
\\[-0.13cm]
\nonumber
&&\mbox{\textit{linear combination of the forms from $\mathcal{L}_i$, 
for $i=1,\ldots,k$.}}
\end{eqnarray}

We note that these conditions are satisfied by binary forms with at least three pairwise non-proportional linear factors, and also discriminant forms, index forms 
and a restricted class of norm forms in an arbitrary number of variables.
As has been explained in \cite[Chap. 9]{evgy15}, conditions \eqref{eq:2a},
\eqref{eq:2b} imply condition (i) of Theorem \ref{B}.

As before, let $S=\{ p_1\kdots p_s\}$ be a finite set of primes, and 
put $P:=\displaystyle{\max_{1\le i\le s} p_i}$. 
Further, let $K$ denote the splitting field of $F$, and put
$d:=[K:\Qq ]$. 
Then we have

\begin{thm}\label{thm1}
Under assumptions \eqref{eq:2a}, \eqref{eq:2b}, we have
\begin{equation}\label{eq:3}
[F(\vec{x})]_S\le \kappa_6 |F(\vec{x})|^{1-\kappa_5} 
\end{equation}
for every $\xv =(x_1\kdots x_m)\in\zmprim$ with $F(\xv )\not=0$,
and with $x_m\ne 0$ if $k>1$, where 
$$
\kappa_5= 
\bigl(c_3^s \bigl((P (\log p_1) \cdots (\log p_s) \bigr)^d \bigr)^{-1} \ge 
(c_3^s (2 P(\log P)^s)^d)^{-1} 
$$ 
and 
$\kappa_6$, $c_3$  are effectively computable positive numbers, 
depending only on $F$. 
\end{thm}

It is easy to check that if $F\in\Zz [X,Y]$ is a binary form with at least
three pairwise non-proportional linear factors over its splitting field,
then it satisfies 
\eqref{eq:2a}, \eqref{eq:2b} with $m=2$ and $k=1$.
Thus, Theorem \ref{thm1b.3} follows at once from Theorem \ref{thm1}.

We shall deduce Theorem \ref{thm1} from a special case of Theorem 3 of Gy\H ory and Yu \cite{gyyu}. The constants $\kappa_5$, $\kappa_6$, $c_3$ 
could have been made explicit by using the explicit version of this theorem of Gy\H ory and Yu \cite{gyyu}. Further, Theorem \ref{thm1} could be proved more generally, over number fields and for a larger class of decomposable forms.

Weaker versions of Theorem \ref{thm1} can be deduced from the results of Gy\H ory \cite{gy1,gy2}.

\section{Proofs of Theorems \ref{thm1a.1}, \ref{thm1a.2}, \ref{thm1b.1}, \ref{thm1b.2}}\label{section3}

Let again $S=\{ p_1\kdots p_s\}$ be a finite, non-empty set of primes.
We denote by $|\cdot |_{\infty}$ the ordinary absolute value,
and by $|\cdot |_p$ the $p$-adic absolute value with $|p|_p=p^{-1}$ 
for a prime number $p$. 
Further, we set $\Qq_{\infty}:=\Rr$, $\overline{\Qq_{\infty}}:=\Cc$.

The following result is a very well-known consequence of the $p$-adic
Thue-Siegel-Roth Theorem. The only reference we could find
for it is \cite[Chap.IX, Thm.3]{mahler61}. For convenience of the reader we recall the proof.

\begin{prop}\label{prop:3.1}
Let $F(X,Y)\in\Zz [X,Y]$ be a binary form of degree $n\geq 2$ and of non-zero
discriminant. Then
\[
\frac{|F(x,y)|}{[F(x,y)]_S}\gg_{F,S,\epsilon} \max (|x|,|y|)^{n-2-\epsilon}
\]
for all $\epsilon >0$ and all $(x,y)\in\zprim$ with $F(x,y)\not= 0$.
\end{prop}

\begin{proof}
We assume that $F(1,0)\not= 0$. This is no loss of generality.
For if this is not the case, there is an integer $b$ of absolute
value at most $n$ with $F(1,b)\not= 0$
and we may proceed with the binary form $F(X,bX+Y)$.
Our assumption implies that
for each $p\in S\cup\{\infty\}$ we have a factorization
$F(X,Y)=a\prod_{i=1}^n(X-\beta_{ip}Y)$ with $a\in\Zz$ and 
$\beta_{ip}\in\overline{\Qq_p}$ algebraic over $\Qq$
for $i=1\kdots n$.
For every $(x,y)\in\zprim$ with $F(x,y)\not= 0$ we have
\begin{eqnarray*}
&&\frac{|F(x,y)|}{[F(x,y)]_S\cdot (\max (|x|,|y|)^n}=\Big(\prod_{p\in S\cup\{ \infty\}} |F(x,y)|_p\Big)/\max (|x|,|y|)^n
\\[0.15cm]
&&\quad \gg_{F,S} \prod_{p\in S}\min_{1\leq i\leq n} \frac{|x-\beta_{ip}y|_p}{\max (|x|_p,|y|_p)}
\\[0.15cm]
&&\quad
\gg_{F,S} 
\prod_{p\in S\cup\{ \infty\}}\min \Big(1,|\medfrac{x}{y}-\beta_{1p}|_p\kdots
|\medfrac{x}{y}-\beta_{np}|_p\Big).
\end{eqnarray*}
The latter is $\gg_{F,S,\epsilon} \max (|x|,|y|)^{-2-\epsilon}$
for every $\epsilon >0$ 
by the $p$-adic Thue-Siegel-Roth Theorem. 
Proposition \ref{prop:3.1} follows.
\end{proof}

\begin{proof}[Proof of Theorem \ref{thm1a.1}]
Let $f(X)\in\Zz [X]$ be the polynomial from Theorem \ref{thm1a.1}.     

(i). The binary form $F(X,Y):=Y^{n+1}f(X/Y)$ has degree $n+1$ and non-zero
discriminant. Now by Proposition \ref{prop:3.1}, we have for every 
$\epsilon >0$ and every sufficiently
large integer $x$,
\[
\frac{|f(x)|}{[f(x)]_S}\gg_{f,S,\epsilon} |x|^{n-1-n\epsilon}\gg_{f,S,\epsilon} |f(x)|^{(n-1-n\epsilon )/n},
\]
implying $[f(x)]_S\ll_{f,S,\epsilon} |f(x)|^{(1/n)+\epsilon}$.

(ii). There are infinitely many primes $p$ such that
$\mod{f(x)}{0}{p}$ is solvable.
Excluding the finitely many primes dividing the
leading coefficient or the discriminant of $f(X)$, 
there remain infinitely many primes.
Take such a prime $p$. By Hensel's Lemma,
there is for every positive integer $k$ an integer $x_k$
such that $\mod{f(x_k)}{0}{p^k}$. We may choose
such an integer with
$p^k\leq x_k<2p^k$.
Then clearly, $x_1<x_2<\cdots$ and for $k$ sufficiently large,
$f(x_k)\not= 0$ and $\mod{f(x_k)}{0}{p^k}$. Consequently,
\[
[f(x_k)]_{\{ p\}}\geq p^k\geq\half |x_k|\gg_f |f(x_k)|^{1/n}.
\]
This proves Theorem \ref{thm1a.1}.
\end{proof}

\begin{proof}[Proof of Theorem \ref{thm1b.1}]
Let $F(X,Y)\in\Zz [X,Y]$ be the binary from Theorem \ref{thm1b.1}.

(i) By Proposition \ref{prop:3.1}, we have for every 
$\epsilon >0$ and every pair $(x,y)\in\zprim$ with $F(x,y)\not= 0$
and $\max (|x|, |y|)$ sufficiently
large,
\[
\frac{|F(x,y)|}{[F(x,y)]_S}
\gg_{F,S,\epsilon} \max (|x|,|y|)^{n-2-n\epsilon}\gg_{F,S,\epsilon} |F(x,y)|^{1-(2/n)-\epsilon}.
\]

(ii) We assume that $F(1,0)\not= 0$ which,  
similarly as in the proof of Proposition \ref{prop:3.1},
is no loss of generality.
By Chebotarev's Density Theorem,
there are infinitely many primes $p$ such that $F$ splits into linear
factors over $\Qq_p$. From these, we exclude the finitely many primes 
that divide $D(F)$ or $F(1,0)$. Let $\PP$ be the infinite set of remaining primes.
Then for every $p\in\PP$, we can express $F(X,Y)$ as
\[
F(X,Y)=a\prod_{i=1}^n(X-\beta_{ip}Y)
\]
with $a\in\Zz$ with $|a|_p=1$, $\beta_{ip}\in\Zz_p$ for $i=1\kdots n$
and $|\beta_{ip}-\beta_{jp}|_p=1$ for $i,j=1\kdots n$ with $i\not=j$. 

We distinguish two cases. First assume that $F$ does not split into
linear factors over $\Qq$. Take $p\in\PP$. Then without loss of generality,
$\beta_{1p}\not\in\Qq$.
Let $k$ be a positive integer. By Minkowski's Convex Body Theorem,
there is a non-zero pair $(x,y)\in\Zz^2$ such that
\[
|x-\beta_{1p}y|_p\leq p^{-k},\ \ \max(|x|,|y|)\leq p^{k/2}.
\]
We may assume without loss of generality that $\gcd (x,y)$ is not divisible
by any prime other than $p$. Assume that $\gcd (x,y)=p^u$ with $u\geq 0$,
and let $x_k:=p^{-u}x$, $y_k:=p^{-u}y$. Then $(x_k,y_k)\in\zprim$ and
\[
|x_k-\beta_{1p}y_k|_p\leq p^{u-k},\ \ \ \max (|x_k|,|y_k|)\leq p^{(k/2)-u}.
\]
This clearly implies $u\leq k/2$. We observe that if we let $k\to\infty$
then $(x_k,y_k)$ runs through an infinite subset of $\zprim$.
Indeed, otherwise we would have a pair $(x_0,y_0)\in\zprim$ with
$|x_0-\beta_{1p}y_0|_p\leq p^{-k/2}$ for infinitely many $k$ which is
impossible since $\beta_{1p}\not\in\Qq$. Next we have $F(x_k,y_k)\not= 0$
for all $k$. Indeed, suppose that $F(x_k,y_k)\not= 0$ for some $k$.
Then $x_k/y_k=\beta_{ip}$ for some $i\geq 2$. Since $\beta_{ip}\in\Zz_p$
we necessarily have $|y_k|_p=1$. But then 
$|x_k-\beta_{1p}y_k|_p=|\beta_{ip}-\beta_{1p}|_p=1$, which is again impossible.
Finally, since clearly $|x_k-\beta_{ip}y_k|_p\leq 1$ for $i=2\kdots n$,
we derive that for each positive integer $k$,
\begin{eqnarray*}
&&[F(x_k,y_k)]_{\{ p\}}=|F(x_k,y_k)|_p^{-1}\geq p^{k-u}
\geq \max (|x_k|,|y_k|)^2
\\
&&\qquad\qquad\qquad\gg_{F,p}
|F(x_k,y_k)|^{2/n}.
\end{eqnarray*}

Next, we assume that $F(X,Y)$ splits into linear factors over $\Qq$.
Then $F(X,Y)=a\prod_{i=1}^n(X-\beta_iY)$ with $a\in\Zz$,
$|a|_p=1$ for $p\in\PP$, $\beta_i\in\Qq$
and $|\beta_i|_p\leq 1$ for $p\in\PP$, $i=1\kdots n$,
and $|\beta_i-\beta_j|_p=1$ for $p\in\PP$, $i,j=1\kdots n$, $i\not= j$.
Pick distinct $p,q\in\PP$ and let $S=\{ p,q\}$. 
Then there is an integer $u$, coprime with $pq$,
such that $u\beta_1$, $u\beta_2$ and $u/(\beta_2-\beta_1)$ are
all integers. Choose positive integers $k,l$.
Then 
\[
x:=\frac{u(\beta_2p^k-\beta_1q^l)}{\beta_2-\beta_1},\ \ 
y:=\frac{u(p^k-q^l)}{\beta_1-\beta_2}
\]
are integers satisfying $x-\beta_1y=up^k$, $x-\beta_2y=uq^l$.
By our choice of $p,q\in\PP$ and by direct substitution,
it follows that the numbers $x-\beta_iy$ $(i=3\kdots n)$
have $p$-adic and $q$-adic absolute values
equal to $1$. Thus,
$|F(x,y)|_p= p^{-k}$,
$|F(x,y)|_q= q^{-l}$ and so $[F(x,y)]_S=p^kq^l$.

Clearly, $g:=\gcd (x,y)$ is coprime with $pq$. Let $x_{k,l}:=x/g$,
$y_{k,l}:=y/g$ so that $(x_{k,l},y_{k,l})\in\zprim$. Then
clearly, $[F(x_{k,l},y_{k,l})]_S=p^kq^l$.
We now choose $k,l$ such that $p^k,q^l$ are approximately equal,
say $p^k<q^l<q\cdot p^k$. 
Then $\max (|x_{k,l}|,|y_{k,l}|)\leq \max (|x|,|y|)\ll_{F,S} (p^kq^l)^{1/2}$ and thus, 
\[
[F(x_{k,l},y_{k,l})]_S\gg_{F,S}\max (|x_{k,l}|,|y_{k,l}|)^2
\gg_{F,S} |F(x_{k,l},y_{k,l})|^{2/n}.
\]
\end{proof}

In the proofs of Theorems \ref{thm1a.2} and \ref{thm1b.2} 
we need a few auxiliary results.

\begin{lem}\label{lem:3.2}
Let $f(X)\in\Zz [X]$ be a polynomial of non-zero
discriminant and $a$ an integer and $p$ a prime.
Denote by $g_p$ the largest non-negative integer $g$ such that
$p^g$ divides the discriminant $D(f)$ of $f$.
For $k>0$ denote by $r(f,a,p^k)$ the number of congruence classes
$x$ modulo $p^k$ with 
$\mod{f(x)}{0}{p^k}$, $\mod{x}{a}{p}$.
Then $r(f,a,p^k)=r(f,a,p^{g_p+1})$ for $k\geq g_p+1$.
\end{lem}

\begin{proof}
This is a consequence of 
\cite[Thm. 2]{stewart}.
\end{proof}

Given a positive integer $h$, we say that two pairs $(x_1,y_1),\, (x_2,y_2)\in\zprim$
are congruent modulo $h$ if $\mod{x_1y_2}{x_2y_1}{h}$.
With this notion, for a given binary form $F(X,Y)\in\Zz [X,Y]$
we can divide the solutions $(x,y)\in\zprim$
of $\mod{F(x,y)}{0}{h}$ into congruence classes modulo $h$.

\begin{lem}\label{lem:3.2a}
Let $F(X,Y)\in\Zz [X,Y]$ be a binary form of degree $n\geq 2$ and of non-zero
discriminant and $p$ a prime. 
Denote by $g_p$ the largest non-negative integer $g$ such that
$p^g$ divides the discriminant $D(F)$ of $F$.
For $k>0$ denote by $r(F,p^k)$ the number of congruence classes modulo $p^k$
of $(x,y)\in\zprim$ with
$\mod{F(x,y))}{0}{p^k}$.
Then $r(F,p^k)=r(F,p^{g_p+1})$ for $k\geq g_p+1$.
\end{lem}

\begin{proof}
Neither the number of congruence classes under consideration, nor the discriminant
of $F$,
changes if we replace $F(X,Y)$ by
$F(aX+bY,cX+dY)$ for some matrix $\big(\begin{smallmatrix}a&b\\c&d\end{smallmatrix}\big)\in{\rm GL}_2(\Zz )$. After such a replacement,
we can achieve that $F(1,0)F(0,1)\not= 0$, so we assume this henceforth.
Let $f(X):=F(X,1)$ and $f^*(X):=F(1,X)$.
The map $(x,y)\mapsto x\cdot y^{-1}\,({\rm mod}\, p^k)$
gives a bijection  
between
the congruence classes modulo $p^k$
of pairs $(x,y)\in\zprim$ with $\mod{F(x,y)}{0}{p^k}$ and $\nmod{y}{0}{p}$
and the congruence classes modulo $p^k$ of integers $z$
with $\mod{f(z)}{0}{p}$.
Likewise, the map $(x,y)\mapsto y\cdot x^{-1}\,({\rm mod}\, p^k)$
establishes a bijection
between the congruence classes modulo $p^k$
of $(x,y)\in\zprim$ with $\mod{F(x,y)}{0}{p^k}$ and $\mod{y}{0}{p}$
and the congruence classes modulo $p^k$ of integers $z$
with $\mod{f^*(z)}{0}{p^k}$ and $\mod{z}{0}{p}$.
Further, our assumption $F(1,0)F(0,1)\not= 0$ implies that 
$D(F)=D(f)=D(f^*)$.
Now an application of Lemma \ref{lem:3.2}
yields that $r(F,p^k)=
\sum_{a=0}^{p-1} r(f,a,p^k)+r(f^*,0,p^k)$
is constant for $k\geq g_p+1$.
\end{proof}

For a binary form $F(X,Y)\in\Rr [X,Y]$ and for positive reals $B,M$, we denote
by $V_F(B,M)$ the set of pairs $(x,y)\in\Rr^2$ with $\max (|x|,|y|)\leq B$
and $|F(x,y)|\leq M$, and by $\mu_F(B,M)$ the area (two-dimensional 
Lebesgue measure) of this set.

Our next lemma is a consequence of a general lattice point counting
result of Barroero and Widmer \cite[Thm. 1.3]{BarWid14}.

\begin{lem}\label{lem:3.2b}
let $n$ be an integer $\geq 2$. Then there is a constant $c(n)>0$ 
such that for every non-zero binary form $F(X,Y)\in\Rr [X,Y]$ of degree $n$,
every lattice $\Lambda\subseteq\Zz^2$ and all positive reals $B,M$,
\[
\left|\#( V_F(B,M)\cap\Lambda )-\frac{\mu_F(B,M)}{\det\Lambda}\right|
\leq c(n)\max (1,B/m(\Lambda )),
\]
where $m(\Lambda )$ is the length of the shortest non-zero vector of $\Lambda$.
\end{lem}

\begin{proof}
We write points in $\Rr^{n+3}\times\Rr^2$ as $(z_0\kdots z_n,u,v,x,y)$.
The set
$Z\subseteq \Rr^{n+3}\times\Rr^2$ given by the inequalities
\[
|z_0x^n+z_1x^{n-1}y+\cdots +z_ny^n|\leq v,\ \ |x|\leq u,\ |y|\leq u
\]
is a definable family in the sense of \cite{BarWid14}, parametrized
by the tuple   
$T=(z_0\kdots z_n,u,v)$. By substituting for this tuple the coefficients of $F$,
respectively $B$ and $M$, we obtain the set $V_F(B,M)$ as defined above.
The sum of the one-dimensional volumes of the orthogonal projections
of $V_F(B,M)$ on the $x$-axis and $y$-axis is at most $4B$,
and the first minimum of $\Lambda$ is $m(\Lambda )$.
Now Lemma \ref{lem:3.2b} follows directly from \cite[Thm. 3.1]{BarWid14}.
\end{proof}

A lattice $\Lambda\subseteq\Zz^2$ is called \emph{primitive} if it contains
points $(x,y)\in\zprim$.

\begin{lem}\label{lem:3.2c}
Let again $n$ be an integer $\geq 2$. Then there is a constant $c'(n)>0$
such that for every binary form $F\in\Zz [X,Y]$ of degree $n$,
every primitive lattice $\Lambda\subseteq\Zz^2$, and all reals $B,M>1$,
\begin{eqnarray*}
&&\left|\#\big( V_F(B,M)\cap\Lambda\cap\zprim\big)-\Big(\frac{6}{\pi^2}\prod_{p|\det\Lambda}
(1+p^{-1})^{-1}\Big)\cdot\frac{\mu_F(B,M)}{\det\Lambda}\right|
\\
&&
\qquad\qquad \leq c'(n)B\log 3B .
\end{eqnarray*}
\end{lem}

\begin{proof}
In the proof below, $p$, $p_i$ denote primes.

Let $F(X,Y)\in\Zz [X,Y]$ be a binary form, $\Lambda\subseteq\Zz^2$ a primitive
lattice, and $B,M$ reals $>1$. Put $d:=\det\Lambda$.
For a positive integer $h$, define the lattice $\Lambda_h:=\Lambda\cap h\Zz^2$. 
Since $\Lambda$ is primitive, there is a basis $\{ \av,\bv\}$ of $\Zz^2$ such that $\{ \av , d\bv\}$
is a basis of $\Lambda$. Hence
$\{ h\av , \lcm (h,d)\bv\}$ is a basis of $\Lambda_h$, and so
\begin{equation}\label{3.1001}
\det\Lambda_h =h\cdot \lcm (h,d)=d\cdot\frac{h^2}{\gcd (h,d)}.
\end{equation}
Further, the shortest non-zero vector of $\Lambda_h$ has length
\begin{equation}\label{3.1002}
m(\Lambda_h)\geq h.
\end{equation} 

We define $\rho (h):=\#(V_F(B,M)\cap\Lambda_h)$. 
Then by the rule of inclusion and exclusion,
\begin{eqnarray*}
&&\#\Big( V_F(B,M)\cap\Lambda\cap\zprim\Big)
\\
&&\quad
=\rho (1)-\sum_{p\leq B} \rho (p)+\sum_{p_1<p_2:\, p_1p_2\leq B} \rho (p_1p_2)-\cdots
\\
&&\quad
=\sum_{h\leq B}\mu (h)\rho (h) ,
\end{eqnarray*}
where $\mu (h)$ denotes the M\"{o}bius function.
The previous lemma together with \eqref{3.1001}, \eqref{3.1002}
implies
\begin{eqnarray*}
&&\left|\#\Big( V_F(B,M)\cap\Lambda\cap\zprim\Big)-
\frac{\mu_F(B,M)}{d}\cdot \sum_{h\leq B}\mu(h)\cdot\frac{\gcd (d,h)}{h^2}\right|
\\
&&
\qquad\qquad
\leq c(n)B\cdot \sum_{h\leq B}\frac{|\mu (h)|}{h},
\end{eqnarray*}
hence
\begin{eqnarray*}
&&\left|\#\Big( V_F(B,M)\cap\Lambda\cap\zprim\Big)-\frac{\mu_F(B,M)}{d}\cdot\sum_{h=1}^{\infty}\mu(h)\cdot\frac{\gcd (d,h)}{h^2}\right|
\\
&&
\qquad\qquad
\leq 
\frac{\mu_F(B,M)}{d}\cdot \sum_{h>B}|\mu (h)|\frac{\gcd (d,h)}{h^2} 
+c(n)B\cdot \sum_{h\leq B}\frac{|\mu (h)|}{h},
\\
&&
\qquad\qquad
\leq c'(n)B\log 3B,
\end{eqnarray*}
where we have used $\sum_{h>B}|\mu (h)|\medfrac{\gcd (d,h)}{h^2}\leq 2d/B$,
$\mu_F(B,M)\leq 4B^2$, and $\sum_{h\leq B}\medfrac{|\mu (h)|}{h}\leq\log 3B$.
Now the proof is finished by observing that
\[
\sum_{h=1}^{\infty}\mu (h)\cdot \frac{\gcd (d,h)}{h^2}
=\prod_{p|d}(1-p^{-1})\cdot\prod_{p\nmid d}(1-p^{-2})
=\frac{6}{\pi^2}\cdot\prod_{p|d}(1+p^{-1})^{-1}.
\]
\end{proof} 

\begin{lem}\label{lem:3.3}
Let $\alpha_1\kdots\alpha_t$ be positive reals. 
Denote by $N(A)$
the number of tuples of non-negative
integers $(u_1\kdots u_t)$ with 
\begin{equation}\label{eq:3.1}
A\leq \alpha_1u_1+\cdots +\alpha_tu_t\leq A+2(\alpha_1+\cdots +\alpha_t).
\end{equation}
Then
\[
N(A)\asymp_{t,\alpha_1\kdots\alpha_t} A^{t-1}\ \ \mbox{as } A\to\infty .
\]
\end{lem}

\begin{proof}
Constants implied by the Vinogradov symbols $\ll$, $\gg$ will depend on
$t,\alpha_1\kdots\alpha_t$.

For $\uv =(u_1\kdots u_t)\in\Zz^t$, denote by $\CC_{\uv}$ the cube in $\Rr^t$
consisting of the points $\yv=(y_1\kdots y_t)$ with $u_i\leq y_i<u_i+1$
for $i=1\kdots t$. Let $\CC$ be the union of the cubes $\CC_{\uv}$ over
all points $\uv$ with non-negative integer coordinates satisfying
\eqref{eq:3.1}. Put $\alpha :=\alpha_1+\cdots +\alpha_t$. 
Then $\CC_1\subseteq \CC\subseteq \CC_2$, where $\CC_1,\CC_2$ are the subsets
of $\Rr^s$ given by 
\begin{eqnarray*}
&&A+\alpha\leq \alpha_1y_1+\cdots +\alpha_ty_t\leq A+2\alpha ,\ \ 
y_1\geq 0\kdots y_t\geq 0,
\\
&&A\leq \alpha_1y_1+\cdots +\alpha_ty_t\leq A+3\alpha,\ \ 
y_1\geq 0\kdots y_t\geq 0,
\end{eqnarray*}
respectively. Clearly $N(A)$ is estimated from below and above
by the measures of $\CC_1$ and $\CC_2$, the first being
$\gg (A+2\alpha )^t-(A+\alpha )^t\gg A^{t-1}$, the second being $\ll A^{t-1}$.
The lemma follows.
\end{proof}

We first give the complete proof of Theorem \ref{thm1b.2}.
The proof of Theorem \ref{thm1a.2} is then obtained by making a few
modifications.

\begin{proof}[Proof of Theorem \ref{thm1b.2}]
Let $F(X,Y)\in \Zz [X,Y]$ be a binary form of degree $n\geq 3$
with non-zero discriminant, $\epsilon$ a real
with $0<\epsilon<\medfrac{1}{n}$ and 
$S=\{ p_1\kdots p_s\}$ a finite set of primes.
Let $S'=\{ p_1\kdots p_{s'}\}$ be the set of $p\in S$ such that $\mod{F(x,y)}{0}{p^{g_p+1}}$
has a solution in $\zprim$, and let $S''=\{ p_{s'+1}\kdots p_s\}$
be the set of remaining primes.
In what follows, constants implied by Vinogradov symbols $\ll$, $\gg$
and by the Landau $O$-symbol will depend only on $F$, $S$ and $\epsilon$.

We first prove that
\[
N(F,S,\epsilon ,B)\ll_{f,S,\epsilon} B^{2-n\epsilon}(\log B)^{s'-1}
\ \ \mbox{as } B\to\infty. 
\]
The set of pairs $(x,y)$ under consideration 
can be partitioned into sets $\NN_h$,
where $h$ runs through the set of positive integers composed of primes from $S$,
and $\NN_h$ is
the set of pairs $(x,y)\in\zprim$
with 
\[
\max (|x|,|y|)\leq B,\ \ [F(x,y)]_S=h,\ \ |F(x,y)|\leq h^{1/\epsilon}.
\]
We first estimate $\#\NN_h$ from above by means of Lemma \ref{lem:3.2c}
where $h$ is any positive integer composed of primes from $S$.
Notice that for $(x,y)\in\NN_h$ we have
$\mod{F(x,y)}{0}{h}$.
By Lemma \ref{lem:3.2a} and the Chinese Remainder Theorem,
the set of these $(x,y)$ lies in $\ll 1$ congruence classes
modulo $h$. Each of these congruence classes
is contained in a set of the shape 
\[
\{ (x,y)\in\Zz^2:\ \mod{y_0x}{x_0y}{h}\}
\]
with $(x_0,y_0)\in\zprim$, which is a 
primitive lattice of determinant $h$. So $\NN_h$ is contained in $\ll 1$
primitive lattices of determinant $h$.

We next estimate the area $\mu_F(B,h^{1/\epsilon})$ of $V(B,h^{1/\epsilon})$.
There is a constant $c_F>0$ such that 
\begin{equation}\label{3.x3}
|F(x,y)|\leq c_F(\max (|x|,|y|)^n\ \ \mbox{for } (x,y)\in\Rr^2.
\end{equation} 
If $h\geq (c_FB^n)^{\epsilon}$ then the condition $|F(x,y)|\leq h^{1/\epsilon}$
is already implied by $\max (|x|,|y|)\leq B$, and so
$\mu_F(B,h^{1/\epsilon})=4B^2$.
On the other hand, if
$h<(c_FB^n)^{\epsilon}$, we have, denoting by $\mu$ the area,
\begin{eqnarray*}
\mu_F(B,h^{1/\epsilon})&\leq& 
\mu\big(\{ (x,y)\in\Rr^2:\, |F(x,y)|\leq h^{1/\epsilon}\}\big)
\\
&=& h^{2/n\epsilon}\cdot\mu\big(\{ (x,y)\in\Rr^2:\, |F(x,y)|\leq 1\}\big)
\ll h^{2/n\epsilon},
\end{eqnarray*}
since the set of $(x,y)\in\Rr^2$ with $|F(x,y)|\leq 1$ has finite area
(see for instance \cite{mahler33}).
Now invoking Lemma \ref{lem:3.2c}, we infer
\begin{equation}\label{3.x4}
\#\NN_h\ll 
\left\{\begin{array}{ll}
B^2/h +O(B\log B)\ \ &\mbox{if $h\geq (c_FB^n)^{\epsilon}$,}
\\
h^{(2/n\epsilon)-1}+O(B\log B) &\mbox{if $h<(c_FB^n)^{\epsilon}$.}
\end{array}\right.
\end{equation}
Finally, from \eqref{3.x3} it is clear that $\NN_h=\emptyset$
if $h>c_FB^n$.

Let $\alpha :=\log (p_1\cdots p_{s'})$.
For $j\in\Zz$, let $\MM_j$ be the union of the sets $\NN_h$ with
\begin{equation}\label{3.x8}
e^{2j\alpha}(c_FB^n)^{\epsilon}\leq h
<e^{(2j+2)\alpha}(c_FB^n)^{\epsilon}.
\end{equation}
We restrict ourselves to $j$ with
\begin{equation}\label{3.x5}
e^{2j\alpha}(c_FB^n)^{\epsilon}\leq c_FB^n,\ \ 
e^{(2j+2)\alpha}(c_FB^n)^{\epsilon}>1, 
\end{equation}
since for the remaining $j$ the set $\MM_j$ is empty. Thus,
\begin{equation}\label{3.x9}
N(F,S,\epsilon ,B)\ll \sum_j \#\MM_j,
\end{equation}
where the summation is over $j$ with \eqref{3.x5}.

We estimate the number of $h$ with \eqref{3.x8}.
Write $h=h'p_1^{u_1}\cdots p_{s'}^{u_{s'}}$ where $h'$ is composed of primes
from $S''$. Then $h'$ divides $\prod_{p\in S''} p^{g_p}$, so
we have $\ll 1$ possibilities for $h'$.
By applying Lemma \ref{lem:3.3} with $t=s'$,  
$A=e^{2j\alpha}(c_FB^n)^{\epsilon})/h'$, $\alpha_i=\log p_i$ for $i=1\kdots s'$,
we infer from Lemma \ref{lem:3.3} that for given $h'$
the number of possibilities for $(u_1\kdots u_{s'})$ is $\ll (\log B)^{s'-1}$.
Hence the number of $h$ with \eqref{3.x8} is $\ll (\log B)^{s'-1}$.
Now from \eqref{3.x4} it follows that for $j$ with \eqref{3.x5},
\[
\#\MM_j\ll 
\left\{
\begin{array}{ll}
e^{-2j\alpha}B^{2-n\epsilon}(\log B)^{s'-1}+O(B(\log B)^{s'})\ \ &
\mbox{if $j\geq 0$,}
\\
e^{-2|j|\alpha ((2/n\epsilon) -1)}B^{2-n\epsilon}(\log B)^{s'-1}+O(B(\log B)^{s'})
\ \ &\mbox{if $j<0$.}
\end{array}\right.
\]
Finally, from these estimates and \eqref{3.x9} we deduce,
taking into consideration that the number of $j$ with \eqref{3.x5}
is $\ll \log B$, and also our assumption $0<\epsilon <\medfrac{1}{n}$, 
\begin{eqnarray*}
N(F,S,\epsilon ,B)
&\ll& \Big(\sum_{j\geq 0} e^{-2j\alpha}+\sum_{j<0} e^{-2|j|\alpha ((2/n\epsilon )-1)}\Big)\cdot B^{2-n\epsilon}(\log B)^{s'-1}
\\
&&\hspace*{6cm}
+O(B(\log B)^{s'})
\\
&\ll& B^{2-n\epsilon}(\log B)^{s'-1}.
\end{eqnarray*}

We next prove that
\[
N(F,S,\epsilon ,B)\gg B^{2-n\epsilon}(\log B)^{s'-1}
\ \ \mbox{as } B\to\infty.
\]

For $i=s'+1\kdots s$, let $a_i$ be the largest integer $u$
such that $\mod{F(x,y)}{0}{p_i^u}$
is solvable in $(x,y)\in\zprim$. 
Let for the moment $h$ be any integer of the shape
$h=p_1^{u_1}\cdots p_s^{u_s}$ 
where $u_i\geq g_{p_i}+1$ for $i=1\kdots s'$ and $u_i=a_i$ for $i=s'+1\kdots s$,
and where $h\geq (c_FB^n)^{\epsilon}$.
By Lemma \ref{lem:3.2c} and the Chinese Remainder Theorem,
the number of congruence classes
modulo $h$ of $(x,y)\in\zprim$ with $\mod{F(x,y)}{0}{h}$ 
is 
\[
r:=\prod_{i=1}^{s'} r(F,p_i^{g_{p_i}+1})\cdot\prod_{i=s'+1}^s r(F,p_i^{a_i}),
\]
which is independent of $h$.
As mentioned above, each of these
congruence classes is contained in a primitive lattice of determinant $h$.
Furthermore, since these lattices arise from different residue classes
modulo $h$ of points in $\zprim$, the intersection of any two of these lattices
does not contain points from $\zprim$ anymore. Since moreover
by our assumption $h\geq (c_FB^n)^{\epsilon}$
the set $V(B,h^{1/\epsilon})$ has area $(4B)^2$, 
an application of Lemma \ref{lem:3.2c} yields that
the set
of $(x,y)\in\zprim$ with $\max (|x|,|y|)\leq B$, $|F(x,y)|\leq h^{1/\epsilon}$
and $\mod{F(x,y)}{0}{h}$ has cardinality
\[
cr\cdot \frac{(4B)^2}{h} +O(B\log B),
\]
where $c=(6/\pi^2)\prod_{p\in S_0}(1+p^{-1})^{-1}$,
with $S_0$ the set obtained from $S$ by removing those primes $p_i$ from $S''$
for which $a_i=0$.
By the rule of inclusion and exclusion, the set $\NN_h$, i.e., the set
of $(x,y)\in\zprim$ as above with $F(x,y)$ divisible by $h$ but not by
$hp$ for $p\in S'$, has cardinality
\begin{eqnarray}\label{3.x7}
&&cr\cdot \frac{(4B)^2}{h}-\sum_{p\in S'}cr\cdot \frac{(4B)^2}{ph}+
\sum_{p,q\in S' , p<q}cr\cdot \frac{(4B)^2}{pqh}-\cdots 
\\
\nonumber
&&\hspace*{8cm} +O(B\log B)
\\
\nonumber
&&= cr\prod_{p\in S'}(1-p^{-1})\cdot\frac{(4B)^2}{h}+O(B\log B)\gg
\frac{B^2}{h}+O(B\log B).
\end{eqnarray}  
We now consider the set of integers $h$
of the shape $p_1^{u_1}\cdots p_s^{u_s}$
with $u_i\geq g_{p_i}+1$ for $i=1\kdots s'$ and $u_i=a_i$
for $i=s'+1\kdots s$, and with 
$(c_FB^n)^{\epsilon}\leq h\leq e^{2\alpha}(c_FB^n)^{\epsilon}$,
where again $\alpha =\log (p_1\cdots p_{s'})$.
By Lemma \ref{lem:3.3}, there are $\gg (\log B)^{s'-1}$ such integers.
Using again $0<\epsilon <\medfrac{1}{n}$, it follows that
\[
N(F,S,\epsilon ,B)\geq \sum_h\#\NN_h\gg B^{2-n\epsilon}(\log B)^{s'-1}.
\]

This completes the proof of Theorem \ref{thm1b.2}.
\end{proof}

\begin{proof}[Proof of Theorem \ref{thm1a.2}]
Let $f\in \Zz [X]$ be a polynomial of degree $n\geq 2$
with non-zero discriminant, $\epsilon$ a real
with $0<\epsilon<\medfrac{1}{n}$ and 
$S=\{ p_1\kdots p_s\}$ a finite set of primes.
Similarly as above
$S'=\{ p_1\kdots p_{s'}\}$ is the set of $p\in S$ such that $\mod{f(x)}{0}{p^{g_p+1}}$
is solvable in $\Zz$ and $S''=\{ p_{s'+1}\kdots p_s\}$.

The proof is the same as that of Theorem \ref{thm1a.2} except from a 
few small modifications. The main difference is that 
instead of Lemma \ref{lem:3.2c} we use the 
simple observation that if $V_f(B,M)$
is the set of $x\in\Rr$ with $|x|\leq B$ and $|f(x)|\leq M$ and $\mu_f(B,M)$
is the one-dimensional measure of this set, then for all $a,h\in\Zz$ with $h>0$,
the number of integers $x\in V_f(B,M)$ with $\mod{f(x)}{a}{h}$
is
\begin{equation}\label{3.x6}
\mu_f(B,M)/h + \mbox{error term},\ \ \mbox{with $|$error term$|\leq c(n)$}
\end{equation}
for some quantity $c(n)$ depending only on $n=\deg f$.

We first prove that
\begin{equation}\label{upperbound}
N(f,S,\epsilon ,B)\ll_{f,S,\epsilon} B^{1-n\epsilon}(\log B)^{s'-1}
\ \ \mbox{as } B\to\infty. 
\end{equation}
Let $c_f$ be a constant such that
$|f(x)|\le c_f|x|^n$ for $x\in\Rr$. Consider the set $\NN_h$ of integers $x$
with $|x|\leq B$, $[f(x)]_S=h$ and $|f(x)|\leq h^{1/\epsilon}$.
Then if $h\geq (c_fB^n)^{\epsilon}$ we have $\mu_f(B,h^{1/\epsilon})=2B$,
while otherwise, $\mu_f(B,h^{1/\epsilon})\ll h^{1/n\epsilon}$,
since $|f(x)|\gg |x|^n$ if $|x|\gg 1$.
Now a similar computation as in the proof of Theorem \ref{thm1b.2},
using Lemma \ref{lem:3.2} instead of Lemma \ref{lem:3.2a}, gives instead
of \eqref{3.x4}, 
\[
\#\NN_h\ll 
\left\{\begin{array}{ll}
B/h +O(1)\ \ &\mbox{if $h\geq (c_fB^n)^{\epsilon}$,}
\\
h^{(1/n\epsilon)-1}+O(1) &\mbox{if $h<(c_fB^n)^{\epsilon}$,}
\end{array}\right.
\]
and then the proof of \eqref{upperbound} is completed in exactly the same
way as in the proof of Theorem \ref{thm1b.2}.

The proof of 
\begin{equation}\label{lowerbound}
N(f,S,\epsilon ,B)\gg_{f,S,\epsilon} B^{1-n\epsilon}(\log B)^{s-1}
\ \ \mbox{as } B\to\infty
\end{equation}
follows the same lines as that of Theorem \ref{thm1b.2}.
For $i=s'+1\kdots s$ let $a_i$ be the largest integer $u$ such that
$\mod{f(x)}{0}{p_i^{a_i}}$ is solvable. Let $h=p_1^{u_1}\cdots p_s^{u_s}$
with $u_i\geq g_{p_i}+1$ for $i=1\kdots s'$ and $u_i=a_i$ for $i=s'+1\kdots s$,
and with $h\geq (c_fB^n)^{\epsilon}$.
Then by combining \eqref{3.x6} with Lemma \ref{lem:3.2} one obtains
that the set of integers $x$ with $|x|\leq B$, $\mod{f(x)}{0}{h}$
and $|f(x)|\leq h^{1/\epsilon}$ has cardinality
\[
rB/h +O(1)
\]
with $r>0$ depending only on $f$, and then an inclusion and exclusion
argument gives
\[
\#\NN_h\gg B/h +O(1).
\]
Again, an argument completely similar to that in the proof of
Theorem \ref{thm1b.2} gives \eqref{lowerbound}.
\end{proof}

\section{Proof of Theorem \ref{thm2.0}}\label{section4}

The theorem can be proved by modifying the arguments from \cite{chenru}.
We prefer to follow \cite[\S8]{ev95}, which already contains
the basic ideas.
Let $F\in\Zz [X_1\kdots X_m]$ be a decomposable form of degree $n$ with splitting field
$K$. We take 
a factorization of $F$ as in \eqref{eq1.0}.
Assume that $F$ satisfies condition (i) of Theorem \ref{B}.

Let $D$ be a linear subspace of $\Qq^m$ of dimension $\geq 2$.
Denote by $D^*$ the $K$-vector space of linear forms in $K[X_1\kdots X_m]$
that vanish identically on $D$.
Then a set of linear forms in $K[X_1\kdots X_m]$ is linearly
dependent on $D$ if some non-trivial $K$-linear combination of these
forms belongs to $D^*$ and linearly independent on $D$ if no such linear combination exists. The $D$-rank $\rankD\MM$ of a set 
of linear forms $\MM\subset K[X_1\kdots X_m]$,
is the maximal number of linear forms in $\MM$
that are linearly independent on $D$. We have $\rankD\LL_F=\dim D$.

We call a subset $\II$ of $\LL_F$ \emph{minimally linearly dependent on $D$,}
if $\II$ itself is linearly dependent on $D$, but every proper, non-empty
subset of $\II$ is linearly independent on $D$.
We define a(n undirected) graph $\GG_D$ as follows. The set of vertices of $\GG_D$ is $\LL_F$;
and $\{ \ell ,\ell'\}$ is an edge of $\GG_D$ if there is a 
subset of $\LL_F$
that is minimally
linearly dependent on $D$ and contains both $\ell$ and $\ell'$.
Clearly, if $\{ \ell ,\ell\}$ is an edge of $\GG_D$, then so is
$\{\sigma (\ell ),\sigma (\ell')\}$ for each $\sigma\in\galk$,
i.e., each $\sigma$ acts on $\GG_D$ as an automorphism.

\begin{lem}\label{lem:4.1}
Let $D$ be a linear subspace of $\Qq^m$ of dimension $\geq 2$ such that
none of the linear forms in $\LL$ vanishes identically on $D$. Then
$\GG_D$ is connected.
\end{lem}

\begin{proof}
Assume that $\GG_D$ is not connected.
Let $\MM$ be the vertex set of a connected component of $\GG_D$.
Then $\emptyset\propersubset\MM\propersubset\LL_F$.
Clearly, for each $\sigma\in\galk$, $\sigma (\MM )$ is also the vertex set
of a connected component of $\GG_D$, hence either $\sigma (\MM )=\MM$,
or $\sigma (\MM )\cap\MM =\emptyset$. That is, $\MM$ is $\galk$-proper.

By assumption (i) of Theorem \ref{B}, the $K$-vector space
\[ 
\sum_{\sigma\in\galk}[\sigma (\MM )]\cap[\LL_F\setminus\sigma (\MM )]
\]
contains a linear form from $\LL$, which, by assumption, does not lie in
$D^*$.
Hence there is $\sigma\in\galk$ such that 
$[\sigma (\MM )]\cap[\LL_F\setminus\sigma (\MM )]$ contains a linear form
outside $D^*$.
But since $D^*$ is defined over $\Qq$,
we have $\sigma (D^*)=D^*$ and so
$[\MM ]\cap[\LL_F\setminus\MM ]$ contains a linear form not in $D^*$,
say $\ell_0$.
Take  maximal subsets $\MM_1$, $\MM_2$
of $\MM$ and $\LL_F\setminus\MM$, respectively, that are both linearly
independent on $D$.
Then there are $\lambda_{\ell}\in K$ for
$\ell\in\MM_1\cup\MM_2$ such that 
\[
\sum_{\ell\in\MM_1}\lambda_{\ell}\ell
\equiv \sum_{\ell\in\MM_2}\lambda_{\ell}\ell\equiv \ell_0\, ({\rm mod}\, D^*).
\]
This implies that
$\MM_1\cup\MM_2$ is linearly dependent on $D$. We can 
take a subset of $\MM_1\cup\MM_2$ that is minimally linearly dependent 
on $D$.
This set necessary must have elements with both $\MM_1$ and $\MM_2$
in common. But then there would be an edge connecting an element
of $\MM$ with one of $\LL_F\setminus\MM$, which 
contradicts that $\MM$ is the vertex set of a connected component
of $\GG_D$.
\end{proof}

\begin{lem}\label{lem:4.2}
Let $D$ be a linear subspace of $\Qq^m$ of dimension $d\geq 2$ 
and $\MM$ a non-empty subset of $\LL_F$ with $\rankD\MM<d$.
Then there is a 
subset $\II$ of $\LL_F$ that is minimally linearly dependent
on $D$, such that
$\MM\cap\II\not=\emptyset$ and $\rankD \MM\cup\II >\rankD\MM$.
\end{lem}

\begin{proof}
Let $\MM'$ consist of all linear forms in $\LL_F$ that are linear combinations
of the linear forms in $\MM$ and of the linear forms in $D^*$. 
Then $\rankD\MM' =\rankD\MM <d$, hence $\emptyset\propersubset\MM'\propersubset\LL_F$.
Take a maximal subset $\MM_1$ of $\MM$ that is linearly independent
on $D$; then it is also
a maximal subset of $\MM'$ that is linearly independent on $D$. 
Let $\MM_2$ be
a maximal subset of $\LL_F\setminus\MM'$ that is linearly
independent on $D$.

By Lemma \ref{lem:4.1}
there is a set $\JJ\subseteq\LL_F$ that is minimally linearly dependent on $D$ and  contains elements of
both $\MM'$ and $\LL_F\setminus\MM'$.
This gives a linear combination $\sum_{\ell\in\JJ}\lambda_{\ell}\ell\in D^*$,
with $\ell_0 :=\sum_{\ell\in\JJ\cap \MM'}\lambda_{\ell}\ell\not\in D^*$.
Writing the linear forms in $\JJ\cap\MM'$ as linear combinations modulo $D^*$
of the linear forms in $\MM_1$, and the linear forms in $\JJ\cap (\LL_F\setminus \MM')$ as linear combinations modulo $D^*$ of the linear 
forms in $\MM_2$, we obtain a relation
$\sum_{\ell\in\MM_1\cup\MM_2}\mu_{\ell}\ell\in D^*$, with
$\sum_{\ell\in\MM_1}\mu_{\ell}\ell\equiv\ell_0\not\equiv 0
\, ({\rm mod}\, D^*)$.
Hence $\MM_1\cup\MM_2$ is linearly dependent on $D$.
Take a subset $\II$ of $\MM_1\cup\MM_2$ that is minimally linearly dependent on $D$. We have
$\II\cap\MM_1\not=\emptyset$ and $\II\cap\MM_2\not=\emptyset$
since $\MM_1$ and $\MM_2$ are linearly independent on $D$.
This implies $\II\cap\MM\not=\emptyset$. Further, $\MM_2\cap\MM'=\emptyset$,
therefore each of the linear forms in $\MM_2$ is linearly independent on $D$ of the linear forms
in $\MM$. Hence $\rankD \MM\cup\II >\rankD\MM$.
\end{proof}

Denote by $M_K$ the set of places of $K$.
We choose normalized absolute values $|\cdot |_v$ $(v\in M_K)$ in such a way that if $v$ lies above $p\in\{\infty\}\cup\{{\rm primes}\}$, then 
$|x|_v=|x|_p^{[K_v:\Qq_p]/[K:\Qq ]}$ for $x\in\Qq$. These absolute values satisfy the product formula $\prod_{v\in M_K} |x|_v=1$ for $x\in K^*$. 
For a vector $\yv=(y_1\kdots y_r)\in K^r$, we define
\[
\|\yv\|_v:=\max_{1\leq i\leq r} |y_i|_v\ \ (v\in M_K),\ \ \ H(\yv ):=\prod_{v\in M_K}\|\yv\|_v.
\]
By the product formula, $H(\lambda\yv )=H(\yv )$ for $\yv\in K^r$, $\lambda\in K^*$.

For $\xv\in\zmprim$ and a subset $\II$ of $\LL_F$, we define 
\[
H_{\II}(\xv ):=\prod_{v\in M_K}\max_{\ell\in\II} |\ell (\xv )|_v.
\]

\begin{lem}\label{lem:4.3}
Let $\xv\in\zmprim$ with $\ell(\xv )\not= 0$ for $\ell\in\LL_F$ and let
$\II$, $\JJ$ be subsets of $\LL_F$ with $\II\cap\JJ\not= \emptyset$. Then
\[
H_{\II\cup\JJ}(\xv )\leq H_{\II}(\xv )\cdot H_{\JJ}(\xv ).
\]
\end{lem}

\begin{proof}
Let $\ell_0\in\II\cap\JJ$. Then by the product formula,
\begin{eqnarray*}
H_{\II\cup\JJ}(\xv )&=&
\prod_{v\in M_K}\max_{\ell\in\II\cup\JJ} |\ell (\xv )/\ell_0(\xv )|_v
\\
&\leq& \Big(\prod_{v\in M_K}\max_{\ell\in\II} |\ell (\xv )/\ell_0(\xv )|_v\Big)\cdot\Big(\prod_{v\in M_K}\max_{\ell\in\JJ} |\ell (\xv )/\ell_0(\xv )|_v\Big)
\\
&=&H_{\II}(\xv )\cdot H_{\JJ}(\xv ).
\end{eqnarray*}
\end{proof}

\begin{lem}\label{lem:4.4}
Let $D$ be a linear subspace of $\Qq^m$ of dimension $\geq 2$ on which
none of the linear forms in $\LL$ vanishes identically.
Then for every $\xv\in\zmprim\cap D$ with $\ell (\xv )\not= 0$ for $\ell\in\LL_F$,
there is a subset $\II$ of $\LL_F$ that is minimally linearly dependent
on $D$ 
such that
$H_{\II}(\xv )\gg_{F,D} \|\xv\|^{1/(m-1)}$.
\end{lem}

\begin{proof} Let $\xv\in\zmprim\cap D$ with $\ell (\xv )\not= 0$
for $\ell\in\LL_F$. 
Start with a linear form $\ell_0\in\LL_F$.
By Lemma \ref{lem:4.2}, there is 
a subset $\II_1$ of $\LL_F$ that is minimally
linearly dependent on $D$, that contains $\ell_0$ and for which $\rankD\II_1\geq 2$. 
Using Lemma \ref{lem:4.2}, we choose inductively 
subsets $\II_2$, $\II_3$, $\ldots$ of $\LL_F$
that are minimally linearly dependent on $D$ as follows:
if $\rankD\II_1\cup\cdots\cup\II_t<\dim D$, we choose $\II_{t+1}$ such that
$\II_{t+1}\cap (\II_1\cup\cdots \cup\II_t)\not=\emptyset$ and
$\rankD\II_1\cup\cdots \cup\II_{t+1}>\rankD\II_1\cup\cdots\cup\II_t$.
It is clear that for some $s\leq \dim D-1\leq m-1$ we get 
$\rankD\II_1\cup\cdots \cup\II_s=\dim D$.
Then 
$X_1\kdots X_m$ can be expressed as linear combinations modulo $D^*$
of the linear forms in $\II_1\cup\cdots \cup\II_s$, implying
$\|\xv\|=H(\xv )\ll_{F,D} H_{\II_1\cup\cdots\cup\II_s}(\xv )$.
Now from Lemma \ref{lem:4.3} we infer
\[
\|\xv\|\ll_{F,D}  \prod_{i=1}^s H_{\II_i}(\xv ) 
\leq
\max_{1\leq i\leq s} H_{\II_i}(\xv )^s\leq \max_{1\leq i\leq s} H_{\II_i}(\xv )^{m-1}.
\]
\end{proof} 

\begin{proof}[Proof of Theorem \ref{thm2.0}]
Without loss of generality, we assume that the linear forms in \eqref{eq1.0}
have their coefficients in the ring of integers $O_K$ of $K$.
We prove by induction on $d$ that if $D$ is a linear subspace
of $\Qq^m$ of dimension $d$ on which none of the linear forms in $\LL$
vanishes identically, then \eqref{eq1.0bb} has only finitely many solutions
in $\zmprim\cap D$. For $d=1$ this is clear.

Assume that $d\geq 2$ and that our assertion holds true for all 
linear subspaces of $\Qq^m$ of dimension smaller than $d$.
Let $D$ be a linear subspace of $\Qq^m$ of dimension $d$ on which none of the
linear forms in $\LL$ vanishes identically and
let $0<\epsilon <\medfrac{1}{m-1}$. 
Take $\xv\in\zmprim\cap D$ satisfying \eqref{eq1.0bb}.
Choose a subset $\II$ of $\LL_F$ that is minimally linearly dependent
on $D$ such that
\begin{equation}\label{eq:4.x}
H_{\II}(\xv )\gg_{F,D} \|\xv\|^{1/(m-1)}.
\end{equation}
Let $T$ be the set of places
of $K$ lying above the places in $\Si$. For $v\in T$, choose $\ell_v\in \II$
such that $|\ell_v(\xv )|_v=\max_{\ell\in\II} |\ell (\xv )|_v$,
and let $\II_v:=\II\setminus\{ \ell_v\}$.
We have $|\ell_v(\xv )|_v\ll_F 1$
for $\ell\in\LL_F$, $v\in M_K\setminus T$ since $\xv\in\Zz^m$. So
by the product formula,
$\prod_{v\in T}|\ell (\xv )|_v\gg_F 1$ for $\ell\in\LL_F$. Together
with \eqref{eq:4.x} this implies
\begin{eqnarray*}
\prod_{v\in T}\prod_{\ell\in\II}|\ell (\xv )|_v
&\ll_F&\prod_{v\in T} |F(\xv )|_v \leq \|\xv\|^{(1/(m-1))-\epsilon}
\\
&\ll_{F,D}& H_{\II}(\xv )^{1-(m-1)\epsilon},
\end{eqnarray*}
and subsequently, dividing both sides by $\prod_{v\in T} |\ell_v(\xv )|_v$,
\begin{equation}\label{eq:4.y}
\prod_{v\in T}\prod_{\ell\in\II_v}|\ell (\xv )|_v\ll_{F,D} H_{\II}(\xv )^{-(m-1)\epsilon}.
\end{equation}
Write $\II =\{ \ell_0\kdots\ell_u\}$.
Then $\ell_0\equiv \beta_1\ell_1+\cdots +\beta_u\ell_u\, ({\rm mod}\, D^*)$ with $\beta_i\not=0$ for $i=1\kdots u$
Put $y_i:=\ell_i(\xv )$ for $i=1\kdots u$,
and $\yv =(y_1\kdots y_u)$. Then $\yv\in O_K^u$.
We can express $\ell (\xv)$ ($\ell\in \II_v$) as $u$ linearly independent
linear forms in $\yv$, say $\ell_{1,v}(\yv )\kdots\ell_{u,v}(\yv )$,
taken from the set $\{ Y_1\kdots Y_u,\beta_1Y_1+\cdots +\beta_uY_u\}$.
Now \eqref{eq:4.y}
translates into
\[
\prod_{v\in T}\prod_{i=1}^u|\ell_{i,v}(\yv )|_v\ll_{F,D} H(\yv)^{-(m-1)\epsilon},\ \ \ \yv\in O_K^u .
\]
Thus, we can apply the $p$-adic Subspace Theorem
\cite{schl77a},
and conclude that the
vectors $\yv$ lie in finitely many proper linear subspaces of $K^u$.
It follows that the solutions $\xv\in\zmprim\cap D$ of \eqref{eq1.0bb},
corresponding to the same sets
$\II_v$ ($v\in T$) in \eqref{eq:4.y}, lie in 
finitely many proper linear subspaces of $D$. Since there are
only finitely many possibilities for the sets $\II_v$, it follows that the solutions
$\xv\in\zmprim\cap D$ 
altogether lie in only finitely many proper linear subspaces of $D$. By applying the induction hypothesis to each of these spaces,
it follows that \eqref{eq1.0bb} has only finitely many solutions in
$\zmprim\cap D$. This completes our proof.
\end{proof}

\section{Proof of Theorem \ref{thm2.1}}\label{section5}

Let $F\in\Zz [X_1\kdots X_m]$ be a decomposable form in $m\geq 2$ variables
with a factorization as in \eqref{eq1.0}, satisfying \eqref{eq1.0e}
and \eqref{eq1.0d}. Our first goal is to prove that $c(F)<1$.
We have used some arguments from \cite[\S3.3]{liu15}.
We start with some preparations.

For a subset $\MM$ of $\LL_F$ we put $|\MM |:=\sum_{\ell\in\MM} e(\ell )$.
Let $D$ be a linear subspace of $\Qq^m$ of dimension $d\geq 2$. 
A subset $\MM$ of $\LL_F$ is called \emph{$D$-critical} if
$q_D(\MM )$ is maximal among all non-empty subsets of $\LL_F$. 
A $D$-critical subset is called \emph{minimal} 
if none of its proper
subsets is $D$-critical.

\begin{lem}\label{lem5.1}
Let $\MM_1$, $\MM_2$ be two $D$-critical subsets of $\LL_F$.
\begin{itemize}
\item[(i)] Assume that $\MM_1$, $\MM_2$ are minimal and $\MM_1\not=\MM_2$.\\
Then $\MM_1\cap\MM_2=\emptyset$.

\item[(ii)] Assume that $\MM_1\cap\MM_2=\emptyset$. Then
$\MM_1\cup\MM_2$ is $D$-critical.
\end{itemize}
\end{lem}

\begin{proof}
We use that for any two subsets $\NN_1$, $\NN_2$ of $\LL_F$ we have
\begin{equation}\label{eq5.1}
\left\{
\begin{array}{l}
\rankD \NN_1\cap\NN_2+\rankD \NN_1\cup\NN_2\leq\rankD\NN_1+\rankD\NN_2,
\\[0.1cm] 
|\NN_1\cap\NN_2|+|\NN_1\cup\NN_2|=|\NN_1|+|\NN_2|.
\end{array}\right.
\end{equation}

(i) Let $q_0:=\max_{\MM} q_D(\MM)$, where the maximum is taken
over all non-empty subsets $\MM$ of $\LL_F$.
Assume $\MM_1\cap\MM_2\not=\emptyset$. Then by \eqref{eq5.1},
\begin{eqnarray*}
&&\rankD\MM_1\cap\MM_2
\\
&&\quad \leq\rankD\MM_1+\rankD\MM_2-\rankD\MM_1\cup\MM_2
\\
&&\quad =q_0^{-1}|\MM_1|+q_0^{-1}|\MM_2|-q_D(\MM_1\cup\MM_2)^{-1}|\MM_1\cup\MM_2|
\\
&&\quad \leq q_0^{-1}(|\MM_1|+|\MM_2|-|\MM_1\cup\MM_2|)=q_0^{-1}|\MM_1\cup\MM_2|,
\end{eqnarray*}
implying $q_D(\MM_1\cap\MM_2)\geq q_0$. This is clearly impossible.

(ii) Again by \eqref{eq5.1}, 
\begin{eqnarray*}
\rankD\MM_1\cup\MM_2&\leq&\rankD\MM_1+\rankD\MM_2
\\
&=&q_0^{-1}(|\MM_1|+|\MM_2|)=q_0^{-1}|\MM_1\cup\MM_2|,
\end{eqnarray*}
which implies $q_D(\MM_1\cup\MM_2)\geq q_0$. Hence $\MM_1\cup\MM_2$
is $D$-critical.
\end{proof}

\begin{lem}\label{lem5.2}
We have $c(F)<1$.
\end{lem}

\begin{proof}
We have to prove that for every $\Qq$-linear subspace $D$ of $\Qq^m$
of dimension $\geq 2$ we have $q_D(F)<q_D(\LL_F)=n/d$, 
where $d=\dim D$ and $n=\deg F$.
Assume that for some of these subspaces $D$ we have $q_D(F)\geq n/d$,
i.e., there is a subset $\MM_1$ of $\LL_F$ with $\rankD\MM_1 <d$
and $q_D(\MM_1)\geq n/d$. Without loss of generality, we take for $\MM_1$
a minimal $D$-critical subset of $\LL_F$.

Let $D^*$ be the $K$-vector space of linear forms in $K[X_1\kdots X_m]$
that vanish identically on $D$. 
Then for each $\sigma\in\galk$, $\sigma (\MM_1 )$ is also a minimal
$D$-critical set since $\rankD\sigma (\MM_1) =\rankD\MM_1$ and 
$|\sigma (\MM_1 )|=|\MM_1 |$, and so by Lemma \ref{lem5.1}, we have either
$\sigma (\MM_1 )=\MM_1$ or $\sigma (\MM_1)\cap\MM_1=\emptyset$.
That is, $\MM_1$ is $\galk$-proper. Let $\MM_1\kdots\MM_t$ be the distinct
sets among the $\sigma (\MM_1)$, $\sigma\in\galk$. We first prove that
\begin{equation}\label{eq5.2} 
\LL_F=\MM_1\cup\MM_2\cup\cdots\cup\MM_t .
\end{equation}
Suppose the contrary, i.e., 
$\MM_0:=\MM_1\cup\cdots\cup\MM_t\propersubset\LL_F$.
By Lemma \ref{lem5.1}, the set $\MM_0$ is $D$-critical,
hence $q_D (\MM_0)\geq d/n$, which implies $\rankD\MM_0<d$.
This, together with the fact that $\MM_0$ is $\galk$-symmetric, 
implies that there is a non-zero $\xv\in D$
with $\ell (\xv )=0$ for $\ell\in\MM_0$. This clearly contradicts \eqref{eq1.0d}.
So indeed, \eqref{eq5.2} holds.
By Lemma \ref{lem5.1} (ii), any non-empty union $\MM$ 
of some of the sets $\MM_i$ is $D$-critical, implying
$q_D(\MM )=q_D(\LL_F)=d/n$. 

As observed above, the set $\MM_1$ is $\galk$-proper. So by assumption
\eqref{eq1.0e}, the $K$-vector space
\[
\sum_{\sigma\in\galk}[\sigma (\MM_1)]\cap[\LL_F\setminus\sigma (\MM_1 )]
=\sum_{i=1}^t [\MM_i]\cap [\LL_F\setminus\MM_i]
\]
contains a linear form from $\LL_F$. By assumption \eqref{eq1.0d},
this form does not lie in $D^*$. Hence there is $i\in\{ 1\kdots t\}$
such that $[\MM_i]\cap [\LL_F\setminus\MM_i]$ contains a linear form
not in $D^*$.
Moreover, $q_D(\MM_i)=q_D(\LL_F\setminus\MM_i)=d/n$. Hence
\begin{eqnarray*}
d=\rankD\LL_F&<&\rankD\MM_i+\rankD (\LL_F\setminus\MM_i)
\\
&=&\medfrac{d}{n}(|\MM_i|+|\LL_F\setminus\MM_i|)=d,
\end{eqnarray*}
which is impossible.
Thus, our assumption that $q_D(F)\geq n/d$ is false.   
\end{proof}

We need a few other, much deeper auxiliary results, which are taken
from the literature. We keep the notation and assumptions from
Theorem \ref{thm2.1}.
For each $p\in\Si$, we choose an extension of $|\cdot |_p$ to the
splitting field $K$ of $F$.

\begin{lem}\label{lem5.3}
Let $D$ be a linear subspace of $\Qq^m$ of dimension $d\geq 2$.
Then for every $\xv\in\zmprim\cap D$, there are subsets $\LL_p$
$(p\in\Si )$ of $\LL_F$ of cardinality $d$ that are linearly independent
on $D$, such that
\begin{eqnarray}\label{eq5.3}
&&\prod_{p\in\Si}\prod_{\ell\in\LL_p} |\ell (\xv )|_p
\\
\nonumber
&&\qquad
\ll_{F,S,D} 
\Big(\big(\prod_{p\in\Si} |F(\xv )|_p\big)\cdot \|\xv\|^{-(n-dq_D(F))}\Big)^{1/q_D(F)}.
\end{eqnarray}
\end{lem}

\begin{proof}
In the case $D=\Qq^m$ this is a special case of \cite[Lemma 2.2.4]{liu15}.
The case of arbitrary $D$ can be reduced to this by choosing a $\Zz$-basis
$\{ \av_1\kdots\av_d\}$ of $\Zz^m\cap D$ and working with the
decomposable form $F(\varphi (\yv))$, where 
$\varphi (\yv)=\sum_{i=1}^d y_i\av_i$.
Note that $\varphi$ establishes a bijection between $\Zz^d_{{\rm prim}}$
and $\zmprim\cap D$.
\end{proof}

\begin{lem}[$p$-adic Minkowski]\label{lem5.4}
Let $p$ be a prime number.
Further, let $\ell_1\kdots\ell_m$ be linearly independent linear forms
in $m$ variables with real coefficients and $\ell_{1,p}\kdots \ell_{m',p}$
$(m'\leq m$) linearly independent linear forms in $m$ variables with coefficients
in $\Qq_p$. Then there are constants $\gamma_1,\gamma_2>1$,
depending only on $p$, $m$, $\ell_1\kdots\ell_m$, $\ell_{1,p}\kdots\ell_{m',p}$,
such that if $A_1\kdots A_m$, $B_1\kdots B_{m'}$ are any positive reals with
\begin{equation}\label{eq5.4}
A_1\cdots A_mB_1\cdots B_{m'}\geq \gamma_1,\ \ \ 
B_i\leq \gamma_2^{-1}\ \mbox{for }i=1\kdots m'
\end{equation}
then there is a non-zero $\xv\in\Zz^m$ with
\begin{equation}\label{eq5.5}
|\ell_i(\xv )|\leq A_i\ \mbox{for } i=1\kdots m,\ \ |\ell_{ip}(\xv )|_p\leq B_i\ \mbox{for }
i=1\kdots m'.
\end{equation}
\end{lem}

\begin{proof}
We augment 
$\ell_{1,p}\kdots\ell_{m',p}$ to a linearly independent set of $m$ linear forms
$\ell_{1,p}\kdots\ell_{m,p}$ with coefficients in $\Qq_p$.
Let $\CC$ be the symmetric convex body consisting of
those $\xv\in\Rr^m$ with 
\[
|\ell_i (\xv )|\leq A_i\ \mbox{for } i=1\kdots m
\]
and $\Lambda$ the lattice
consisting of those $\xv=(x_1\kdots x_m)\in\Qq^m$
such that
\begin{eqnarray*}
&&|\ell_{ip}(\xv )|_p\leq B_i\ \mbox{for } i=1\kdots m',
\\
&&|\ell_{i,p}(\xv )|_p\leq\gamma_2^{-1}\ \mbox{for } i=m'+1\kdots m,
\\
&&|x_i|_q\leq 1\ \mbox{for $i=1\kdots m$ and all primes $q\not= p$,}
\end{eqnarray*}
with $\gamma_2$ yet to be chosen.
By choosing
$\gamma_2$ sufficiently large, we can guarantee that $\Lambda\subseteq\Zz^m$
for all $B_1\kdots B_{m'}\leq \gamma_2^{-1}$ and by choosing $\gamma_1$
sufficiently large, we can guarantee that 
${\rm vol}(\CC )/\det\Lambda\geq 2^m$ for all $A_1\kdots A_m$
with $A_1\cdots A_mB_1\cdots B_{m'}\geq \gamma_1$. Minkowski's Theorem
implies that for such $A_i,B_i$ there is a non-zero $\xv\in\CC\cap\Lambda$. 
This $\xv$ satisfies \eqref{eq5.5} and lies in $\Zz^m$. 
\end{proof}

\begin{prop}\label{prop5.5}
Let $F\in\Zz [X_1\kdots X_m]$ be a decomposable form 
of degree $n$ with \eqref{eq1.0e}
and \eqref{eq1.0d}. Then 
the number of $\xv\in\zmprim$ with $\prod_{p\in\Si} |F(\xv )|_p\leq M$
is $\ll_{n,S} M^{m/n}$ as $M\to\infty$.
\end{prop}

\begin{proof} 
Liu proved this in his thesis for all decomposable forms $F$ with 
$c(F)<1$ and with \eqref{eq1.0d}, see \cite[Theorem 2.1.3]{liu15}.
As observed in Lemma \ref{lem5.2}, the condition $c(F)<1$ follows
from \eqref{eq1.0e} and \eqref{eq1.0d}.
Liu's theorem and its proof are a $p$-adic
generalization of Thunder's theorem \cite[Theorem 2]{thunder01} 
and its proof.
\end{proof}

\begin{proof}[Proof of Theorem \ref{thm2.1}]
(i). Let $0<\epsilon <1-c(F)$. We prove by induction on $d$ that if $D$ is any 
$d$-dimensional $\Qq$-linear subspace 
of $\Qq^m$, then $[F(\xv )]_S\ll_{F,S,D} |F(\xv )|^{c(F)+\epsilon}$ for all
$\xv\in\zmprim\cap D$. For $d=1$ this is clear. Let $d\geq 2$,
and assume the assertion is true for all linear subspaces of $\Qq^m$
of dimension $<d$. Let $D$ be a $\Qq$-linear subspace of $\Qq^m$
of dimension $d$. Take $\xv\in\zmprim\cap D$ for which
\begin{equation}\label{eq5.6}
[F(\xv )]_S\geq |F(\xv )|^{c(F)+\epsilon}.
\end{equation}
Then 
\[
\prod_{p\in\Si}|F(\xv )|_p =\frac{|F(\xv )|}{[F(\xv )]_S}
\leq |F(\xv )|^{1-c(F)-\epsilon}
\ll_{F,S,D} \|\xv\|^{n(1-c(F)-\epsilon)}.
\]
Take subsets $\LL_p$ ($p\in\Si$) of $\LL_F$ as in Lemma \ref{lem5.4}
and insert the above inequality into \eqref{eq5.3}.
Then since $c(F)\geq q_D(F)\cdot d/n$,
\begin{align*} 
\prod_{p\in\Si}\prod_{\ell\in\LL_p} |\ell (\xv )|_p\,
&\ll_{F,S,D} 
\Big(\|\xv\|^{n(1-c(F)-\epsilon )}\cdot \|\xv\|^{-(n-dq_D(F))}\Big)^{1/q_D(F)}
\\
&\ll_{F,S,D} \|\xv\|^{-n\epsilon/q_D(F)}.
\end{align*}
By the $p$-adic Subspace Theorem, the points $\xv\in\zmprim\cap D$ with
\eqref{eq5.6} lie in finitely many proper linear subspaces of $D$. By applying
the induction hypothesis with each of these subspaces, we infer that for
the points $\xv\in\zmprim\cap D$ with \eqref{eq5.6} we have 
$[F(\xv )]_S\ll_{F,S,D} |F(\xv )|^{c(F)+\epsilon}$. 
This completes our induction step, and hence the proof of (i).

(ii). Let $K=\Qq (\theta )$. By  
Chebotarev's Density Theorem there are infinitely many primes $p$
such that the minimal polynomial of $\theta$ over $\Qq$ has all its roots
in $\Qq_p$. Take such a prime $p$. Then in the factorization \eqref{eq1.0}
we may assume that the linear forms in $\LL_F$ have their
coefficients in $\Qq_p$. Let $D$ be a linear subspace of $\Qq^m$
of dimension $d\geq 2$, and $\MM$ a subset of $\LL_F$ with $\rankD\MM=:d'<d$
for which $q_D(\MM )\cdot d/n=c(F)$. 
Choose a subset $\MM'$ of $\MM$ of cardinality $d'$ that is linearly
independent over $D$.
By Lemma \ref{lem5.4} there is for 
every sufficiently large $Q$ a non-zero point $\xv\in\Zz^m\cap D$ such that
\[
\|\xv\|\ll Q,\ \ |\ell (\xv )|_p\ll Q^{-d/d'}\ \ \mbox{for } \ell\in\MM',
\]
where here and below, the constants implies by $\ll$ depend on $F,D$ and $p$ and in fact
only on $F$ and $p$ since $D$ depends on $F$.
Without loss of generality, we may assume that the greatest common divisor
of the coordinates of $\xv$ does not contain factors coprime with $p$.
Let $p^k$ be the greatest common divisor of the coordinates of $\xv$
and put $\xv':=p^{-k}\xv$
Then $\xv'\in\zmprim\cap D$, $p^k\ll Q$, and
\begin{equation}\label{eq5.7}
\begin{array}{rl}
\|\xv'\|&\ll\, p^{-k}Q,
\\[0.1cm]
|\ell (\xv')|_p&\ll\, p^kQ^{-d/d'}=(p^{-k}Q)^{-d/d'}(p^k)^{1-(d/d')}
\ \ \mbox{for } \ell\in\MM' .
\end{array}
\end{equation}
Now if we let $Q\to\infty$, then $\xv'$ runs through an infinite set.
Indeed,
otherwise there were a non-zero $\xv'\in\zmprim\cap D$
such that \eqref{eq5.7} holds for arbitrarily large $Q$.
But by letting $Q\to\infty$, we can make $\max (p^{-k}Q,p^k)$ 
arbitrarily large and thus $|\ell (\xv ')|_p$ arbitrarily small
for every $\ell\in\MM'$. But then it would follow
that $\ell (\xv')=0$ for $\ell\in\MM'$, which is however excluded by
assumption \eqref{eq1.0d}.

From the above we conclude that there are infinitely many $\xv'\in\zmprim\cap D$
such that
\[
|\ell (\xv')|_p\ll \|\xv'\|^{-d/d'}\ \ \mbox{for } \ell\in\MM'.
\]
Since the other linear forms in $\MM$ are linear combinations modulo $D^*$
of the linear forms in $\MM'$ , these $\xv'$ satisfy
\[
|\ell (\xv')|_p\ll \|\xv'\|^{-d/d'}\ \ \mbox{for } \ell\in\MM ,
\]  
and moreover, trivially, $|\ell (\xv')|_p\ll 1$ for $\ell\in\LL_F\setminus\MM$.
Using the decomposition \eqref{eq1.0}, it follows that these $\xv'$ 
satisfy
\[
|F(\xv')|_p\ll \|\xv'\|^{-(d/d')|\MM |} =\|\xv'\|^{-dq_D(\MM )}=\|\xv'\|^{-nc (F)},
\]
hence
\[
[F(\xv')]_{\{p\}}=|F(\xv')|_p^{-1}\gg |F(\xv')|^{c(F)}.
\]
This proves (ii).

(iii) Let $0<\epsilon <1$ and $B>1$. Then every $\xv\in\zmprim$ with 
$[F(\xv )]_S\geq |F(\xv )|^{\epsilon}$ and $\|\xv\|\leq B$ satisfies
\[
\prod_{p\in\Si} |F(\xv )|_p=\frac{|F(\xv )|}{[F(\xv )]_S}\leq |F(\xv )|^{1-\epsilon}
\ll_{F,\epsilon} B^{n(1-\epsilon )},
\]
where $n:=\deg F$. Hence $N(F,S,\epsilon ,B)$ is at most the number of solutions 
in $\xv\in\zmprim$ of this last inequality.
Now Proposition \ref{prop5.5} implies 
\[ 
N(F,S,\epsilon ,B)\ll_{F,S,\epsilon} (B^{n(1-\epsilon )})^{m/n}\ll_{F,S,\epsilon} B^{m(1-\epsilon )}
\]
as $B\to\infty$. This proves (iii).
\end{proof}

\section{Proof of Theorem \ref{thm1}}\label{section6}
Theorem \ref{thm1} will be deduced from Proposition \ref{prop:1} below, which is a special case of a non-explicit version of Theorem 3 of Gy\H ory and Yu \cite{gyyu}. Its proof is based on effective results of Gy\H ory and Yu \cite{gyyu} for unit equations, and ultimately depends on Baker's method, more precisely on explicit estimates of Matveev \cite{matv}
concerning linear forms in 
complex 
logarithms of algebraic numbers
and similar such estimates by Yu \cite{yu} for $p$-adic logarithms.

Let $F\in\Zz [X_1\kdots X_m]$ be a decomposable form, 
$S=\{ p_1\kdots p_s\}$ a finite non-empty set of primes, and $b$ 
a non-zero integer. Let 
$\Z_S :=\Z[(p_1\cdots p_s)^{-1}]$ 
be the ring of $S$-integers in $\Q$, and consider the equation
\begin{equation}\label{eq:5}
F(\vec{x})=b\hspace{3mm}\textrm{in}\hspace{3mm}\vec{x}\in\Z^m_S.
\end{equation}
Let $\mathfrak{p}_1,\ldots,\mathfrak{p}_t$ be the prime ideals in $K$ 
that divide $p_1,\ldots,p_s$, and let $P'=\displaystyle{\max_{1\le i\le t}} N(\mathfrak{p}_i)$,
where $N(\fa ):=\# O_K/\fa$ denotes the absolute norm of a non-zero
ideal $\fa$ of $O_K$.
Further, denote by $h$ the absolute logarithmic height.

\begin{prop}\label{prop:1}
Let $F$ be a decomposable form as above with properties \eqref{eq:2a} and 
\eqref{eq:2b}.
With the above notation, every solution $\vec{x}=(x_1,\ldots,x_m)\in\Z_S^m$ of \eqref{eq:5} with $x_m\ne 0$ if $k>1$ satisfies
\begin{equation}
\begin{aligned}\label{eq:6}
\max_{1\le j\le m} h(x_j)&<c_4^t(P'/\log P')\prod_{i=1}^t \log N(\mathfrak{p}_i)\cdot\\
&\qquad\qquad\cdot(c_5+\log N(\mathfrak{p}_1 \cdots\mathfrak{p}_t)+h(b)),
\end{aligned}
\end{equation}
where $c_4$, $c_5$ are effectively computable positive numbers that depend only on $F$.
\end{prop}

We mention that Theorem 3 of \cite{gyyu} implies Proposition \ref{prop:1} with explicit
expressions for $c_4$, $c_5$ in terms of the heights of the coefficients of $F$
and the degree and regulator of the splitting field $K$ of $F$.  

We now prove Theorem \ref{thm1} by means of Proposition \ref{prop:1}.

\begin{proof}[Proof of Theorem \ref{thm1}]
Let $\xv\in\zmprim$ with $F(\xv )\not= 0$, and put $b:=F(\xv )/[F(\xv )]_S$.
Then $F(\xv )=p_1^{a_1}\cdots p_s^{a_s}b$ for certain non-negative integers $a_1\kdots a_s$.
We can write $a_i=na_i'+a_i''$ with $a_i'$, $a_i''\in\Z_{\ge 0}$ such that $a_i''<n$ for each $i$. Then \eqref{eq:1} implies that
\begin{equation}
F(\vec{x}')=b',\label{eq:7}
\end{equation}
where
\begin{equation}\label{eq:8}
\vec{x}'=\vec{x}/p_1^{a_1'}\cdots p_s^{a_s'}\hspace{3mm}\textrm{and}\hspace{3mm} b'=bp_1^{a_1''}\cdots p_s^{a''_s}.
\end{equation}
Here $\vec{x}' = (x'_1, \ldots , x'_m)$ is a solution of \eqref{eq:7} in $\Z_S^m$.

We apply now Proposition \ref{prop:1} to the equation \eqref{eq:7}. Let $\mathfrak{p}_1,\ldots,\mathfrak{p}_t$ and $P'$ as above. Then we get
$$\max_{1\le j\le m} h(x'_j) < C_1,$$
for every solution $\vec{x}=(x_1',\ldots,x_m')\in\Z_S^m$ of \eqref{eq:7} 
with $x_m'\ne 0$ if $k>1$, 
where $C_1$ denotes the upper bound occuring in \eqref{eq:6} but with $b$ replaced by $b'$.

Since $t\le sd$, $P'\le P^d$ where $d=[K:\Qq ]$, and $h(b')\le ns\log P+\log|b|$, 
 we infer that 
\begin{equation}
\max_{1\le j\le m} h(x_j') < C_2(c_6+\log|b|),\label{eq:9}
\end{equation}
where $C_2=c_7^s(P (\log p_1) \cdots (\log p_s))^d$ and $c_6$, $c_7$ are effectively computable positive numbers that depend only on $F$. It is easy to deduce from \eqref{eq:9} and \eqref{eq:8} that
$$
p_1^{a_1'}\cdots p_s^{a_s'}\le C_3|b|^{mC_2},
$$
where $C_3=\e^{mC_2c_6}$. This gives
$$p_1^{a_1}\cdots p_s^{a_s}<(p_1\cdots p_s)^n(p_1^{a_1'}\cdots p_s^{a_s'})^n\le C_4|b|^{mnC_2}$$
with $C_4=P^{sn}C_3^n$. Multiplying both sides by $(p_1^{a_1}\cdots p_s^{a_s})^{mnC_2}$
and then raising to the power $1/(mnC_2+1)$, 
we infer that
$$[F(\vec{x})]_S\le (P^sC_3)^{\frac{1}{mC_2}}|F(\vec{x})|^{1-\frac{1}{mnC_2+1}}.$$
But $(P^sC_3)^{\frac{1}{mC_2}}\le \kappa_6$,
while $mnC_2+1\leq c_3^s(P (\log p_1) \cdots (\log p_s))^d$ with
effectively computable $\kappa_6 ,c_3$ depending only on $F$. 
This gives \eqref{eq:3}.
\end{proof}

\section{Lower bound for the greatest prime factors of decomposable forms at integral points}\label{section7}

We now deduce over $\Z$ an improved and more explicit version of Corollary 5 of Gy\H ory and Yu \cite{gyyu} on the greatest prime factors of decomposable forms at integral points. We note that in Gy\H ory and Yu \cite{gyyu} it was more complicated to deduce Corollary 5 from Theorem 3 of that paper. The next corollary gives some useful information about those non-zero integers that can be represented by decomposable forms of the above type.

For a positive integer $a$ we denote by $P(a)$ and $\omega(a)$ the greatest prime factor and the number of distinct prime factors of $a$ with the convention that $P(1)=1$, $\omega(1)=0$. Further, we denote by $\log_i$ the $i$-th iterated logarithm.

\begin{cor}\label{cor2}
	Let $F(X_1\kdots X_m)\in\Z[X_1,\ldots,X_m]$ be a decomposable form as in Theorem \ref{thm1}, and let $F_0$ be a non-zero integer that can be represented by $F(\vec{x})$ with some $\vec{x}=(x_1\kdots x_m)\in\zmprim$ with $x_m\not= 0$ if $k>1$. Then
	\begin{equation}\label{eq:10}
	(P(\log P)^{2\omega})^d>\log |F_0|
	\end{equation}
	and
	\begin{equation}\label{eq:11}
	P>\begin{cases}
	(\log |F_0|)^{1/3d}\textrm{ if }\omega\le \log P/\log_2 P,\\
	C_5\log_2 |F_0| \cdot \log_3 |F_0|/\log_4 |F_0| \textrm{ otherwise,}
	\end{cases}
	\end{equation}
	provided that $|F_0|\ge C_6$, where $P=P(F_0)$, $\omega=\omega(F_0)$. Here $C_5$, $C_6$ are effectively computable positive numbers that depend only on $F$.
\end{cor}

\begin{proof}
	Let $F_0$ be a non-zero integer such that $F_0=F(\vec{x})$ for some $\vec{x}=(x_1,\ldots,x_m)\in\zmprim$ with $x_m\ne 0$ if $k>1$. Write
	$$F(\vec{x})=p_1^{a_1}\cdots p_s^{a_s}$$
	with distinct primes $p_1,\ldots,p_s$. Then $P=P(F_0)=\displaystyle{\max_{1\le i\le s} p_i}$ and $\omega=\omega(F_0)=s$. Put $S:=\{p_1,\ldots,p_s\}$. 
	In this case $[F(\vec{x})]_S=|F(\vec{x})|$. Now \eqref{eq:3} immediately gives
	$$|F_0|\le \kappa_6  |F_0|^{1-\kappa_5}$$
	with $\kappa_5$, $\kappa_6$ specified in Theorem \ref{thm1}. This implies that
	$$|F_0|\le \kappa_6^{1/\kappa_5},$$
	whence
	$$\log |F_0|\le c_8^s(2 P(\log P)^s)^d$$
	with an effectively computable positive $c_8$ that depends only on $F$.
	
	We know from prime number theory that $s<\medfrac{2P}{\log P}$. Hence, if $|F_0|\ge C_7$ with a large and effectively computable $C_7=C_7(F)>0$, then $P$ must be also large and so $(c_8(\log P)^d)^s\le(\log P)^{2ds}$ and \eqref{eq:10} follows.
	
	If $s\le\medfrac{\log P}{\log_2 P}$ then it follows from \eqref{eq:10} that
	$$\log_2 |F_0|<d\log P+2ds\log_2 P\le 3d\log P,$$
	which gives the first inequality in \eqref{eq:11}, provided that $C_7$ is sufficiently large.
	Otherwise, we deduce from \eqref{eq:10} that 
	$$\log_2 |F_0|<d\log P+4 d \frac{P}{\log P} \log_2 P,$$
	which gives the second inequality in \eqref{eq:11}, provided that $C_7$ is sufficiently large.
\end{proof}

\section{Applications to discriminants of algebraic integers}\label{section8} 

As was mentioned above, Theorem \ref{thm1} and its corollaries can be applied to  discriminant forms, index forms and a large class of norm forms. We now present some applications to discriminants of algebraic integers. Similar consequences can be obtained for indices of algebraic integers.

Let $L$ be a number field of degree $n\ge 3$ with ring of integers $O_L$, and suppose that $K$ is the normal closure of $L$ over $\Q$.
Further, let $S=\{ p_1\kdots p_s\}$ be a finite, non-empty set
of primes. We define the discriminant of an algebraic integer to be the 
discriminant of its monic minimal polynomial over $\Zz$.
Consider the \emph{discriminant equation}
\begin{equation}\label{eq:12}
D_{L/\Qq}(\alpha)=p_1^{a_1}\cdots p_s^{a_s}\cdot b\textrm{ in } \alpha\in O_L,
\ \ a_1\kdots a_s\in\Zz_{\geq 0},
\end{equation}
where $b$ is an $S$-free integer, i.e.,  coprime with $p_1\kdots p_s$.
Clearly, $\alpha$ and $\alpha+a$ with $a\in\Z$ have the same discriminant. Such elements of $O_L$ are called \textit{equivalent}. 
Denote by $\mathscr{S}$ the set of positive integers 
composed of primes from $S$.
We claim that any solution of \eqref{eq:12} can be derived from one which is not equivalent to any element of $O_L$ that is divisible in $O_L$
by 
any 
$\eta>1$ from $\mathscr{S}$. Indeed, if $\alpha$ satisfies \eqref{eq:12} then, by Theorem 3 of Gy\H ory \cite{gy3}, $\alpha$ can be written in the form
$$\alpha=\eta\alpha'+a$$
with some $a\in\Z$, $\eta\in\mathscr{S}$ and $\alpha'\in O_L$. This representation is not necessarily unique. For fixed $\alpha$, choose $\eta$, $\alpha'$, $a$ such that $\eta$ is maximal. Since $D_{L/\Qq}(\alpha)=\eta^{n(n-1)}D_{L/\Qq}(\alpha')$, $\alpha'$ is also a solution of $\eqref{eq:12}$ with other $a_1,\ldots,a_s$. Further, by the choice of $\eta$, the number $\alpha'$ cannot be equivalent to any $\eta'\alpha''$ in $O_L$ with $\alpha''\in O_L$ and $\eta'\in\mathscr{S}$ with $\eta'>1$, since otherwise $\alpha$ would be equivalent to $\eta\eta'\alpha''$ with $\eta\eta'>\eta$. This proves our claim.

Note that in the representation \eqref{eq:12}, the $S$-part of the discriminant of
$\alpha$ is 
$$[D_{L/\Qq}(\alpha)]_S=p_1^{a_1}\cdots p_s^{a_s}.
$$
As a consequence of Theorem \ref{thm1}, we want to estimate $[D_{L/\Qq}(\alpha)]_S$ from above in terms of $|D_{L/\Qq}(\alpha)|^{1-\kappa_7}$ for some constant $\kappa_7>0$.
In view of the above we require that $\alpha$ not be equivalent to any element of the form $\eta\alpha'$  where $\alpha'\in O_L$ and $\eta$ is an integer from $\mathscr{S}$ with $\eta>1$.

\begin{cor}\label{cor4}
	Assume that $\alpha$ in \eqref{eq:12} is not equivalent to any element of $O_L$ that is divisible in $O_L$ by an $\eta\in\mathscr{S}$ greater than $1$. Then
	\begin{equation}\label{eq:13}
	[D_{L/\Qq}(\alpha)]_S\le \kappa_8\cdot  |D_{L/\Qq}(\alpha )|^{1-\kappa_7},
	\end{equation}
	where
	$$\kappa_7= 
\bigl(c_9^s \bigl((P (\log p_1) \cdots (\log p_s) \bigr)^d \bigr)^{-1} \ge 
(c_9^s (2 P(\log P)^s)^d)^{-1} ,$$
	and $\kappa_8$, $c_9$ are effectively computable positive numbers depending only on $L$.
\end{cor}

\begin{proof}
If $L$ is effectively given in the sense of e.g., Evertse and Gy\H ory \cite[\S3.7]{evgy}, an integral basis of $O_L$ of the form
$\{1,\omega_2,\ldots,\omega_n\}$ can be effectively determined. 
Then we can write $\alpha=a+x_2\omega_2+\cdots +x_n\omega_n$ with appropriate integers $a,x_2,\ldots,x_n$.
Using the fact that $D_{L/\Qq}(\alpha)=D_{L/\Qq}(\alpha-a)$ we get
\[	
D_{L/\Qq}(\alpha )=D_{L/\Qq}(x_2\omega_2+\cdots +x_n\omega_n).
\]
By the assumption made on $\alpha$ we infer that $p_1\kdots p_s$ do not divide 
$\gcd(x_2,\ldots,x_n)$. Moreover, we may assume without loss of generality that $\gcd(x_2,\ldots,x_n)=1$. The discriminant form $D_{L/\Qq}(\omega_2X_2+\cdots+\omega_nX_n)$ satisfies 
\eqref{eq:2a} and \eqref{eq:2b} with $k=1$, see e.g. Gy\H ory and Yu \cite{gyyu},
so we can apply Theorem \ref{thm1} with this discriminant form. By observing that the dependence of the constants in Theorem \ref{thm1} 
can be replaced by a dependence on $L$, Corollary \ref{cor4} follows.
\end{proof}

Corollary \ref{cor4} has similar consequences as Theorem \ref{thm1} for arithmetical properties of non-zero integers $D_0$ that are discriminants of some $\alpha\in O_L$, but are not the discriminants of any $k\beta$ with $\beta\in O_L$ and rational integer $k>1$. Then it follows from Theorem \ref{thm1} that
$$(P(\log P)^{2\omega})^d>|D_0|$$
provided that $|D_0|\ge C(L)$, where $P=P(D_0)$, $\omega=\omega(D_0)$ and $C(L)$ is effectively computable in terms of $L$. We can get also inequalities similar to \eqref{eq:11}. We note that more general but weaker results of this type can be found in Gy\H ory \cite{gy3} and Evertse and Gy\H ory \cite{evgy}.

\section{Additional comments}\label{section-last}

Let $f(X)$ be an integer polynomial with at least two distinct roots,
and $S=\{ p_1\kdots p_s\}$ a finite set of primes. According to the result
of Gross and Vincent \cite{grovin} quoted as Theorem \ref{A} 
in
the Introduction, we have 
$$
[f(x)]_S\leq \kappa_2|f(x)|^{1-\kappa_1}\ \ 
\mbox{for every $x\in\Zz$ with $f(x)\not= 0$,}
$$
where $\kappa_1,\kappa_2$ are positive numbers, effectively computable in terms
of $f$ and $S$. As mentioned in Theorem \ref{thm1a.3}, in this estimate
we can take 
$$
\kappa_1 = \bigl(c_1^s \bigl(( (\max_i p_i) (\log p_1) \cdots (\log p_s) \bigr)^d \bigr)^{-1},
$$
where $d$ is the degree of the splitting field of $f$ 
and $c_1$ depends only on $f$. 
The factor $\max_i p_i$ comes from the use of linear forms in $p$-adic logarithms
in our argument.
If we follow instead the
proof of \cite{grovin}, by applying a result of Matveev \cite{matv} replacing
the older 
and less sharp estimate for linear forms in logarithms
due to Alan Baker that was used by Gross and Vincent, 
we would have obtained an estimate of the above type with
$$
\kappa_1 = \bigl(c_{2}^s \bigl((\log p_1) \cdots (\log p_s)  \bigr)^{c_{3}} \bigr)^{-1},
$$
where $c_{2}, c_{3}$ (as well as the other constants $c_{4}$, $c_{5},\ldots , c_8$
below) are effectively computable in terms of $f$.
Taking for $p_1, \ldots , p_s$ the first $s$ prime numbers, an easy computation using
the Prime Number Theorem shows that, for every positive $\epsilon$, 
we have
$$
P(f(x)) \ge (1 - \epsilon ) \,  \log_2 x \cdot \log_3 x/\log_4 x\, ,
$$
for $x\in\Zz$ with $f(x)\not= 0$ and $|x|$ sufficiently large in terms of $\epsilon$. 

For a positive integer $a$ we denote by $Q(a)$ its greatest square-free factor. 
Let again $x$ be an integer with $f(x)\not= 0$
and $p_1\kdots p_s$ the prime divisors of $f(x)$.
Proceeding as in \cite{grovin}, but applying a result of Matveev \cite{matv} 
instead of one of Baker, we get
$$
\log |x| \leq c_{4}^s  \bigl((\log p_1) \cdots (\log p_s)  \bigr)^{c_{5}}.
$$
Using the arithmetico-geometric inequality as in Stewart's paper \cite{ste08},
we deduce that
$$
\frac{\log \log |x|}{s} \le c_{6} \, \Bigl( 1 + \log \Bigl( \frac{\log Q(f(x))}{s} \Bigr) + 
\frac{\log_3 Q(f(x))}{s} \Bigr).
$$
We then conclude that
$$
\log Q(f(x)) \ge c_{7} \,  \log_2 |x|\cdot \log_3 |x|/\log_4 |x|.
$$
With the approach followed in the present paper, we would only get that 
$$
\log Q(f(x)) \ge c_{8} \,  \log_2 |x|, 
$$
that was already known. 

Let $F\in\Z[X_1,\ldots,X_m]$ be a decomposable 
form as in Theorem \ref{thm1}, and let $F_0$ be a non-zero integer 
that can be represented by $F(\vec{x})$ with some 
$\vec{x}=(x_1\kdots x_m)\in\zmprim$ with $x_m\not= 0$ if $k>1$. 
We are not able to prove the existence 
of effectively computable positive numbers $c_9, c_{10}$, which depend only on $F$,
such that
$$
\log Q(F_0) > c_9 \, \log_2 |F_0| \cdot \log_3 |F_0|/\log_4 |F_0|,
$$
provided that $|F_0| > c_{10}$.


\begin{thebibliography}{99}

\bibitem{BarWid14}
F. Barroero and M.Widmer,
{\it Counting Lattice Points and O-Minimal Structures},
Int. Math. Res. Not. {\bf 2014:18}, 4932--4957 (2014).

\bibitem{BeFiTr08}
M. A. Bennett, M. Filaseta, and O. Trifonov, 
{\it Yet another generalization of the Ramanujan-Nagell equation}, 
Acta Arith. {\bf 134}, 211--217 (2008).


\bibitem{BeFiTr09}
M. A. Bennett, M. Filaseta, and O. Trifonov, 
{\it On the factorization of consecutive integers}, 
J. Reine Angew. Math. {\bf 629}, 171--200 (2009).
 


\bibitem{chenru}
Z. Chen and M. Ru, 
\textit{Integer solutions to decomposable form inequalities}, J. Number Theory \textbf{115}, 58--70 (2005).

\bibitem{ev84}
J.-H. Evertse,
\textit{On sums of $S$-units and linear recurrences}, 
Compos. Math. {\bf 53}, 225--244 (1984).

\bibitem{ev95}
J.-H. Evertse,
{\it The number of solutions of decomposable form equations},
Invent. Math. {\bf 122}, 559--601 (1995).

\bibitem{evgy88}
J.-H. Evertse and K. Gy\H ory,
\textit{Finiteness criteria for decomposable form equations}, Acta Arith. {\bf 50}, 357--379 (1988).
     
\bibitem{evgy15}
J.-H. Evertse and K. Gy\H ory, \textit{Unit Equations in Diophantine Number Theory}, Cambridge University Press (2015).

\bibitem{evgy}
J.-H. Evertse and K. Gy\H ory, \textit{Discriminant Equations in Diophantine Number Theory}, Cambridge University Press (2016).

\bibitem{grovin}
S. Gross and A. Vincent, \textit{On the factorization of $f(n)$ for $f(x)$ in $\Z[x]$}, Int. J. Number Theory \textbf{9}, 
1225--1236 (2013).

\bibitem{gy1}
K. Gy\H ory, \textit{Explicit upper bounds for the solutions of some diophantine equations}, Ann. Acad. Sci. Fenn., Ser. A I, Math. \textbf{5}, 3--12 (1980).

\bibitem{gy2}
K. Gy\H ory, \textit{On the representation of integers by decomposable forms in several variables}, Publ. Math. Debrecen \textbf{28}, 89--98 (1981).

\bibitem{gy3}
K. Gy\H ory, \textit{On discriminants and indices of integers of an algebraic number field}, J. Reine Angew. Math. \textbf{324}, 114--126 (1981).

\bibitem{gypapp}
K. Gy\H ory and Z. Z. Papp, \textit{Effective estimates for the integer solutions of norm form and discriminant form equations}, Publ. Math., Debrecen \textbf{25}, 311--325 (1978).

\bibitem{gyyu}
K. Gy\H ory and K. Yu, \textit{Bounds for the solutions of $S$-unit equations and decomposable form equations}, Acta Arith. \textbf{123}, No. 1, 9--41 (2006).


\bibitem{liu15}
J. Liu, {\it On p-adic decomposable form inequalities}, Ph.D. thesis, Leiden (2015).

\bibitem{mahler33}
K. Mahler, {\it Zur Approximation algebraischer Zahlen. III. (\"{U}ber die mittlere Anzahl der Darstellungen grosser Zahlen durch bin\"{a}re Formen)},
Acta Math. {\bf 62} (1933), 91--166.

\bibitem{mahler61}
K. Mahler, {\it Lectures on Diophantine approximations. part I: $g$-adic numbers and Roth's theorem},
Univ. Notre Dame, 1961.

\bibitem{matv}
E. M. Matveev, \textit{An explicit lower bound for a homogeneous rational linear form in the logarithms of algebraic numbers. II.},  Izv. Ross. Akad. Nauk, Ser. Mat. \textbf{64} No. 6 125-180 (2000) (Russian), translation: Izv. Math. \textbf{64}, No. 6, 1217--1269 (2000).



\bibitem{vdpschl82}
A.J. van der Poorten and H.P. Schlickewei,
{\it The growth condition for recurrence sequences}, Macquarie University Math. Rep. 82-0041 (1982).

\bibitem{vdpschl91}
A.J. van der Poorten and H.P. Schlickewei,
{\it Additive relations in fields},
J. Austral. Math. Soc. (Ser. A) {\bf 51}, 154--170 (1991).

\bibitem{schl77a}
H.P. Schlickewei, 
{\it The $\fp$-adic Thue-Siegel-Roth-Schmidt Theorem},
Arch. Math. {\bf 29}, 267--270 (1977).

\bibitem{schl77}
H.P. Schlickewei,
{\it On norm form equations},
J. Number Theory {\bf 9}, 370--380 (1977).
  	
\bibitem{stewart}
C.L. Stewart, \textit{On the number of solutions of polynomial congruences and Thue equations}, J. Amer. Math. Soc. \textbf{4}, 
793--835 (1991).
	
\bibitem{ste08} 
C.L. Stewart,
{\it On the greatest square-free factor of terms of a linear recurrence sequence}. 
In: Diophantine equations, 257--264, 
Tata Inst. Fund. Res. Stud. Math., 20, Tata Inst. Fund. Res., Mumbai, 2008.

\bibitem{thunder01}
J.L. Thunder,
{\it Decomposable form inequalities},
Ann. of Math. {\bf 153}, 767--804 (2001).

\bibitem{yu}
K. Yu, \textit{$p$-adic logarithmic forms and group varieties. III.}, Forum Math. \textbf{19}, No. 2, 187--280 (2007).
\end{thebibliography}
\end{document}